\documentclass[a4paper,11pt]{article}

\usepackage[hyperfootnotes = true]{hyperref}

\usepackage{amsmath}
\usepackage{graphicx}
\usepackage{psfrag}
\usepackage{epsf}
\usepackage{enumerate}
\usepackage[round]{natbib}
\usepackage{url} 
\usepackage{amsthm}
\usepackage{xpatch}
\usepackage{subcaption}

\usepackage{longtable}
\usepackage{mathrsfs}
\usepackage{tikz}
\usepackage{color}
\usepackage{amssymb}
\usepackage{latexsym}
\usepackage{xr}
\usepackage{epstopdf}
\usepackage{float}
\usepackage{multirow}
\usepackage{dsfont}
\usepackage{mathabx}

\numberwithin{equation}{section}

\usepackage[autostyle]{csquotes}
\usepackage{sectsty}
\subsectionfont{\normalfont\itshape}
\subsubsectionfont{\normalfont\itshape}

\def\RR{\mathbb R}
\def\ZZ{\mathbb Z}

\def\NN{\mathbb N}

\newcommand{\E}{\operatorname{E}} 
\newcommand{\Var}{\operatorname{Var}} 
\newcommand{\Cov}[0]{\operatorname{Cov}}

\newcommand{\vecop}{\operatorname{vec}}

\newcommand{\diag}{\operatorname{diag}}

\newcommand{\VPsi}{\operatorname{\Psi}} 

\newcommand{\Hi}{\mathbb{H}}

\newtheorem{lemma}{Lemma}[section]

\newtheorem{theorem}[lemma]{Theorem}

\newtheorem{proposition}[lemma]{Proposition}

\newtheorem{remark}[lemma]{Remark}

\makeatletter
\xpatchcmd{\proof}{\@addpunct{.}}{\@addpunct{:}}{}{}
\makeatother

\DeclareFontFamily{U}{mathx}{\hyphenchar\font45}
\DeclareFontShape{U}{mathx}{m}{n}{<-> mathx10}{}
\DeclareSymbolFont{mathx}{U}{mathx}{m}{n}
\DeclareMathAccent{\widebar}{0}{mathx}{"73}

\usepackage{color}
\usepackage{marginnote}

\usepackage{geometry}
\begin{document}

\def\spacingset#1{\renewcommand{\baselinestretch}%
{#1}\small\normalsize} \spacingset{1}


%

\title{{\bf Limit theorems in the context of multivariate long-range dependence}}

\author{
Marie-Christine D\"uker\hyperlink{myth}{\parbox{3pt}{\footnotemark[1]}}}

\maketitle

\let\oldthefootnote\thefootnote
\renewcommand{\thefootnote}{\fnsymbol{footnote}}
\footnotetext[1]{\phantomsection\hypertarget{myth}{Ruhr-Universit\"at Bochum, Faculty of Mathematics, Bochum, Germany, 
\protect \url{Marie-Christine.Dueker@rub.de}}}
\let\thefootnote\oldthefootnote

%
%

\maketitle

\begin{abstract}
\noindent
This article considers multivariate linear processes whose components are either short- or long-range dependent.
The functional central limit theorems for the sample mean and the sample autocovariances for these processes are investigated, paying special attention to the mixed cases of short- and long-range dependent series. The resulting limit processes can involve multivariate Brownian motion marginals, operator fractional Brownian motions and matrix-valued versions of the so-called Rosenblatt process.


\end{abstract}

%
%
%
%
%
%
%
%
%
%

\section{Introduction}
In this work, we are interested in the asymptotic behavior of the sample mean and the sample autocovariances for multivariate linear processes under several assumptions on their dependence structure, with the focus on long-range dependence. \par
For long-range dependent time series, the autocovariance decays power-like as the time lag increases.
Limit theorems for univariate long-range dependent time series were studied by a number of authors. See, for example, \cite{Dav70, HorvathKokoszka, Taqqu1975} and for an overview of long-range dependence, \cite{beran2013long, giraitis, PipirasTaqqu}.
Limit theorems for multivariate processes under long-range dependence were studied in \cite{chung_2002, Dai2013, Dai2017, Dejong2000}. Vectors of univariate long-range dependent time series were considered in \cite{BaiTaqqu, BaiTaqqu2013}. 
\par
The present work develops limit theorems for two-sided multivariate linear processes, whose components are allowed to be either short- or long-range dependent.
Under short-range dependence, the dependence parameters take values in a range different from long-range dependence, so that the autocovariances decay faster for the corresponding components. 
We investigate functional limit theorems for these processes for the sample mean, as well as for the sample autocovariances, i.e. we prove weak convergence in a multivariate product space of $D[0,1]$, the space of c\`{a}dl\`{a}g functions on $[0,1]$ equipped with the uniform metric.
Depending on the dependence structure, the limit involves Brownian motion marginals, operator fractional Brownian motions and matrix-valued versions of the so-called Rosenblatt process.
\par
Our setting is as follows. We consider an $p$-dimensional second-order stationary time series $\{X_{n}\}_{n\in \ZZ}$ with $X_{n}=(X_{1,n},\dots,X_{p,n})'$, where the prime indicates transposition.
We suppose throughout the paper that $X_{n}$ can be represented as a multivariate linear process
\begin{equation}	\label{equality_general_linear_process}
X_{n} = \sum_{j\in \ZZ} \VPsi_{j}\varepsilon_{n-j},
\end{equation}
where $\{\VPsi_{j}\}_{j \in \ZZ}$ is a sequence of matrices with $\VPsi_{j}=(\psi_{kl,j})_{k,l=1,\dots,p}\in \RR ^{p \times p}$ and 
$\{\varepsilon_{j}\}_{j\in \ZZ}$ is a sequence of mean zero independent and identically distributed (i.i.d.) random vectors with covariance matrix $\E(\varepsilon_{0}\varepsilon_{0}')=I_{p}$. 
Write the entries of the matrices $\{\VPsi_{j}\}_{j \in \ZZ}$ as
\begin{equation} \label{equality_long_range_dep_linear_process}
\psi_{kl,j}=C_{kl}(j)|j|^{d_{k}-1}, \hspace*{0.2cm} j \in \ZZ^{*},
\end{equation}
for $k,l \in \{1,\dots, p\}$, where $\ZZ^{*}=\ZZ\setminus\{0\}$. 
We further assume that $p_{1}$ components of \eqref{equality_general_linear_process} are multivariate long-range dependent and $p_{2}:=p-p_{1}$ components are multivariate short-range dependent in the following sense.
The first $p_{1}$ and the last $p_{2}$ components of the linear process \eqref{equality_general_linear_process} are called, respectively,
\begin{itemize}
\item[(L)] multivariate long-range dependent, if $d_{k} \in (0,\frac{1}{2})$ for $k \in \{1,\dots,p_{1} \}$, and $C_{kl}(j)$ in \eqref{equality_long_range_dep_linear_process} satisfies
\begin{equation*}	
C_{kl}(j) \sim \alpha^{+}_{kl} 
\hspace*{0.2cm} \text{ as } \hspace*{0.2cm} j \to \infty
\hspace*{0.2cm} \text{ and } \hspace*{0.2cm}
C_{kl}(j) \sim \alpha^{-}_{kl}
\hspace*{0.2cm} \text{ as } \hspace*{0.2cm} j \to -\infty,
\end{equation*}
where the matrices 
\begin{equation*}
A^{+}=(\alpha^{+}_{kl})_{k=1,\dots,p_{1};l=1,\dots,p},
\hspace{0.2cm}
A^{-}=(\alpha^{-}_{kl})_{k=1,\dots,p_{1};l=1,\dots,p} \in \RR^{p_{1} \times p} 
\end{equation*}
are assumed to have full rank; 
\item[(S)] multivariate short-range dependent, if $d_{k}<0$ for 
$k \in \{p_{1}+1,\dots, p\}$, and $C_{kl}(j)$ in \eqref{equality_long_range_dep_linear_process} satisfies \begin{equation*}
\sup_{ j \in \ZZ^{*} } | C_{kl}(j) | \leq \beta
\end{equation*}
for all $k \in \{p_{1}+1,\dots, p\}$, $l \in \{1,\dots,p\}$ and some constant $\beta >0$ and the matrix
\begin{equation*}
\sum_{j \in \ZZ} (\psi_{kl,j})_{k=p_{1}+1,\dots,p;l=1,\dots,p} \in \RR^{p_{2} \times p}
\end{equation*}
is assumed to have full rank. 
\end{itemize}
The definition (L) allows for quite general long-range dependent, linear time series, for example, multivariate FARIMA series; see \cite{KechagiasPipiras}.
For univariate time series, the short-range dependence definition (S) was introduced in \cite{BaiTaqqu}. It allows for exponentially decaying coefficients $\psi_{kl,j}$. The multivariate setting includes, for example, vector ARMA models.
\par
Under the introduced setting, we are interested in the asymptotic behavior of the vector-valued sample mean process
\begin{equation} \label{equality_partial_sum_process}
S_{\lfloor Nt \rfloor} =
\sum_{n=1}^{\lfloor Nt \rfloor} X_{n}, \hspace*{0.2cm} t \in [0,1],
\end{equation}
and the asymptotic behavior of the sample autocovariance process
\begin{equation}	\label{equality_autocovariance_process}
((\widehat{\Gamma}_{N,\ell}-\Gamma_{\ell})(t), \ell=0,\dots,L), \hspace*{0.2cm} t \in [0,1],
\end{equation}
with
\begin{equation} \label{equality_Gammahat}
(\widehat{\Gamma}_{N,\ell}-\Gamma_{\ell})(t)
=		
\frac{1}{N}\sum_{n=1}^{\lfloor Nt \rfloor} 
\left( X_{n} X_{n+\ell}'-\E(X_{0} X_{\ell}') \right),
\hspace*{0.2cm}
\widehat{\Gamma}_{N,\ell}(t)
=		
\frac{1}{N}\sum_{n=1}^{\lfloor Nt \rfloor} X_{n} X'_{n+\ell}.
\end{equation}
In order to prove convergence of the sample mean process under different assumptions on the dependence structure, we first consider a more general result, which is of independent interest. It states that the linear process in \eqref{equality_general_linear_process} satisfies the central limit theorem when $\sum_{j\in \ZZ} \| \VPsi_{j} \|_{F}^{2} < \infty$, where $\| \cdot \|_{F}$ denotes the Frobenius norm. 
In the univariate case, this result was proven in \citet[Theorem 18.6.5]{IbragimovLinnik}.
\par
Multivariate processes under less flexible assumptions on the dependence structure were studied by \cite{chung_2002, Dejong2000}, who considered one-sided multivariate long-range dependent linear processes and derived limit theorems for the sample mean and sample autocovariances in terms of the convergence of the finite-dimensional distributions (f.d.d.).
The works \cite{Dai2013} and \cite{Dai2017} characterized a class of processes which converge to an operator fractional Brownian motion, and the latter also considered
limit theorems for functionals of Gaussian vectors.
In \cite{BaiTaqqu2013}, limits of a vector of normalized sums of functions of long-range dependent stationary Gaussian series were studied and \cite{BaiTaqqu} investigated the limit of normalized partial sums of a vector of multilinear polynomial-form processes. 
The two latter works also allowed for \enquote{mixture} cases, where vectors of both univariate short- and long-range dependent time series were studied.
\par
The works \cite{BaiTaqqu, BaiTaqqu2013} are perhaps the closest to this study. 
However, our work is based on a linear process generated by a multivariate i.i.d.\ sequence and allows for multivariate short- and long-range dependence. Considering the sample autocovariance matrix of such a process leads to a matrix-valued process, whose entries depend on different combinations of short- and long-range dependent parameters.
In contrast, \cite{BaiTaqqu,BaiTaqqu2013} considered vectors of univariate processes. 
The dependence structure in \cite{BaiTaqqu2013} is determined by the Hermite rank of the respective function applied to a univariate long-range dependent process. 
The work \cite{BaiTaqqu} supposed that each component can be represented as a univariate multilinear polynomial form process obtained by applying an off-diagonal multi-linear polynomial-form filter to an i.i.d.\ sequence. They allowed the components to be either short- or long-range dependent. 
See Remarks \ref{remark_comp_smp} and \ref{remark_comp2} for more details. 
\par
The rest of the paper is organized as follows. In Section \ref{section_2}, some properties of multivariate short- and long-range dependent time series are reviewed and the processes resulting as limits of \eqref{equality_partial_sum_process} and \eqref{equality_autocovariance_process} are given.
In Section \ref{section_3}, we present the limit theorems concerning the sample mean process \eqref{equality_partial_sum_process}.
In Section \ref{section_4}, the functional limit theorems for the sample autocovariances \eqref{equality_autocovariance_process} are presented.
In the last section, we provide the detailed proofs.

\section{Preliminaries} \label{section_2}
In this section, we introduce some further notation, give more details about short- and long-range dependence and define the resulting limit processes.
\par
The autocovariances of a second-order stationary zero mean time series $\{X_{n}\}_{n \in \ZZ}$ at lag $\ell$ are denoted and defined as 
\begin{equation*}
			\Gamma_{\ell}
	=		(\gamma_{kl}(\ell))_{k,l=1,\dots,p}
	=		\E(X_{0} X_{\ell}')
	=		\sum_{j\in \ZZ} \VPsi_{j} \VPsi_{j+\ell}'.
\end{equation*}
The autocovariance $\Gamma_{\ell}$ takes values in the space of real-valued matrices $\RR^{p \times p}$ equipped with the Frobenius inner product 
$\langle A, B \rangle=\sum_{k,l=1}^{p} a_{kl}b_{kl}$ for 
$A=(a_{kl})_{k,l=1,\dots,p}$, $B= (b_{kl})_{k,l=1,\dots,p}$, which induces the Frobenius norm $ \| A \|_{F}^{2}=\langle A, A \rangle=\sum_{k,l=1}^{p} |a_{kl}|^2$.
The convergence of the finite-dimensional distributions and that in law are denoted by $\overset{f.d.d.}{\longrightarrow}$ and $\overset{\mathcal{L}}{\longrightarrow}$, respectively. We use the notation $D[0,1]^{p}$ for the $p$-dimensional product space of $D[0,1]$.
Furthermore, we write $a^{G}=\operatorname{diag}(a^{g_{1}},\dots,a^{g_{p}})$ for $a > 0$ and a diagonal matrix $G=\operatorname{diag}(g_{1},\dots,g_{p})$. Also, we set 
$D=\diag(d_{1},\dots,d_{p})$,
$D_{p_{1}}=\diag(d_{1},\dots,d_{p_{1}})$ and 
$D^{c}_{p_{2}}=\diag(d_{p_{1}+1},\dots,d_{p})$ with $0 \leq p_{1} \leq p$, $p_{2} = p-p_{1}$ and $D_{0}=D^{c}_{p+1}=0$
for the dependence parameters introduced in \eqref{equality_long_range_dep_linear_process}. 
\par
In proving our convergence results, we will use the following two propositions on autocovariances of linear processes satisfying the long- or short-range dependence definition (L) and (S) in Section 1. For this purpose, let $\{X^{\operatorname{L}}_{n}\}_{n \in \ZZ}$ denote the $p_{1}$-dimensional process satisfying (L) and $\{X^{\operatorname{S}}_{n}\}_{n \in \ZZ}$ the $p_{2}$-dimensional process satisfying (S), which combined together make the process $\{X_{n}\}_{n \in \ZZ}$. The respective autocovariances are denoted by
\begin{equation*}
\Gamma_{p_{1},\ell}=\E(X^{\operatorname{L}}_{0}X^{\operatorname{L}'}_{\ell})
\hspace*{0.7cm}
\text{ and }
\hspace*{0.7cm}
\Gamma_{p_{2},\ell}=\E(X^{\operatorname{S}}_{0}X^{\operatorname{S}'}_{\ell}).
\end{equation*}

\begin{proposition} \label{Proposition_mLRD}
The autocovariances of the process $\{X^{\operatorname{L}}_{n}\}_{n \in \ZZ}$ satisfy
\begin{equation}	\label{equality_Definition_KP_multivariate_stationary_lrd}
\Gamma_{p_{1},\ell}=		
\ell^{D_{p_{1}}-\frac{1}{2}I_{p_{1}}}R(\ell)\ell^{D_{p_{1}}-\frac{1}{2}I_{p_{1}}}=	
(R_{kl}(\ell)\ell^{d_{k}+d_{l}-1})_{k,l=1,\dots, p_{1}},	
\end{equation}
where $R(\ell)= (R_{kl}(\ell))_{k,l=1,\dots,p_{1}}$ is a function satisfying
\begin{equation*}
R(\ell) \sim R=(R_{kl})_{k,l=1,\dots, p_{1}} \text{ as } \ell \to \infty
\end{equation*}
with
\begin{equation}	\label{equation_Rij}
R_{kl}=		
\operatorname{B}(d_{k},d_{l})
\left( c_{1,kl} \frac{\sin(\pi d_{k})}{\sin(\pi (d_{k}+d_{l}))}+
c_{2,kl}+c_{3,kl} \frac{\sin(\pi d_{l})}{\sin(\pi (d_{k}+d_{l}))} \right) ,
\end{equation}
where $\operatorname{B}$ denotes the beta function and
\begin{equation*} 
c_{1,kl}	=(A^{-}(A^{-})')_{kl},
\hspace*{0.2cm}
c_{2,kl}	=(A^{-}(A^{+})')_{kl},
\hspace*{0.2cm}
c_{3,kl}	=(A^{+}(A^{+})')_{kl}.
\end{equation*}
\end{proposition}
\begin{proof}
The proof is given in \citet[Proposition 3.1]{KechagiasPipiras}.
\end{proof}

\begin{proposition}	\label{Proposition_mSRD}
The autocovariances of the process $\{X^{\operatorname{S}}_{n}\}_{n \in \ZZ}$ are absolutely 
summable in the sense that 
\begin{equation*}	
\sum_{\ell \in \ZZ} \| \Gamma_{p_{2},\ell} \|_{F} < \infty.
\end{equation*}
\end{proposition}

\begin{proof}
Note that $\Gamma_{p_{2},-\ell}=\Gamma_{p_{2},\ell}'$.
As in the proof of \citet[Proposition 3.1]{KechagiasPipiras}, one has
\begin{equation*}
|\Gamma_{p_{2},\ell}| \leq \ell^{D^{c}_{p_{2}}-\frac{1}{2}I_{p_{2}}} T(\ell) \ell^{D^{c}_{p_{2}}-\frac{1}{2}I_{p_{2}}}, 
\hspace{0.2cm}
\ell>0,
\end{equation*}
for some $\RR^{p_{2} \times p_{2}}$-valued function $T(\ell)$, whose components are slowly varying functions. Indeed, note that
\begin{equation*}
|\gamma_{kl}(\ell)| \leq \ell^{d_{k}+d_{l}-1} 
\sum_{i=1}^{3} T_{i,kl}(\ell),
\hspace{0.2cm}
\ell>0,
\end{equation*}
where for example,
\begin{equation*}
T_{1,kl}(\ell)=
p \beta^2
\sum_{j=\ell+1}^{\infty} \Big(\frac{j}{\ell} \Big)^{d_{k}-1} 
\Big(\frac{j}{\ell}-1\Big)^{d_{l}-1}\frac{1}{\ell}.		
\end{equation*}
Then, there is a constant $C$ such that
\begin{equation*}
\sum_{\ell \in \ZZ} \|\Gamma_{p_{2},\ell}\|_{F} \leq
2C \sum_{\ell=0}^{\infty}  
\left( \sum_{k,l=p_{1}+1}^{p} \left| 
\ell^{d_{k}+d_{l}-1}\right| ^2\right) ^{\frac{1}{2}}.
\end{equation*}
The absolute summability follows, since $1-d_{k}-d_{l}>1$.
\end{proof}
\par
We now turn to the processes resulting as limits of the sample mean and autocovariance processes. 
In connection to the sample mean process, we follow \cite{DidierPipiras2011} to introduce operator fractional Brownian motions (OFBMs). OFBMs are multivariate extensions of the univariate fractional Brownian motion and denoted here as
$\mathcal{B}^{(p)}_{H}(t)=(\mathcal{B}_{1,H}(t),\dots,\mathcal{B}_{p,H}(t))' \in \RR^{p}$ with $ t \in \RR$ and some symmetric matrix $H \in \RR^{p \times p}$. They are Gaussian, operator self-similar with exponent $H$ and have stationary increments.
Additionally, it shall be assumed that they are proper, that is,
for each $t \in \RR$, the distribution of $\mathcal{B}^{(p)}_{H}(t)$ is not contained in a proper subspace of $\RR^{p}$.
The process $\{\mathcal{B}^{(p)}_{H}(t)\}_{t \in \RR}$ is operator self-similar 
(\cite{HudsonMason1982,Laha1981}) if $\{\mathcal{B}^{(p)}_{H}(ct)\}_{t \in \RR} \overset{f.d.d.}{=} \{c^{H}\mathcal{B}^{(p)}_{H}(t)\}_{t \in \RR}$ for every $c>0$.
To introduce an integral representation for OFBMs, let $W(dx)=(W_{1}(dx),\dots,W_{p}(dx))$ be a multivariate, real-valued Gaussian random measure satisfying
\begin{equation*}
\E W(dx)=0,
\hspace*{0.7cm}
\E W(dx)W'(dx)=I_{p} dx,	
\hspace*{0.7cm}
\E W_{k}(dx)W_{l}(dy)=0, \hspace*{0.2cm} x \neq y.
\end{equation*}
Then, if the eigenvalues of the symmetric matrix $H$ denoted by $h_{k}$ satisfy $0<h_{k}<1$ and $h_{k} \neq \frac{1}{2}$ for $k=1,\dots, p$, in the time domain, the OFBM $\mathcal{B}^{(p)}_{H}(t)$ admits the representation 
$\mathcal{B}^{(p)}_{H}(t)\overset{f.d.d.}{=} \mathcal{B}^{(p)}_{H,M^{+},M^{-}}(t)$ with
\begin{equation*}
\mathcal{B}^{(p)}_{H,M^{+},M^{-}}(t)
=
\int_{\RR} 
\Big(((t-x)^{H-\frac{1}{2}I}_{+}-(-x)_{+}^{H-\frac{1}{2}I}) M^{+} + 
((t-x)^{H-\frac{1}{2}I}_{-}-(-x)_{-}^{H-\frac{1}{2}I}) M^{-} \Big) W(dx),
\end{equation*}
where $x_{+}=\max(0,x)$, $x_{-}=\max(-x,0)$, $ M^{+},  M^{-} \in \RR^{p \times p}$; see \cite{DidierPipiras2011}.
According to \cite{LAVANCIER}, the corresponding cross-covariance function is given by
\begin{equation}  \label{equality_cov_OFBMS}
\E \mathcal{B}^{(p)}_{H}(t) \mathcal{B}^{(p)'}_{H}(u) =	
|t|^{H}\widetilde{R}|t|^{H}+|u|^{H}\widetilde{R}'|u|^{H}-|t-u|^{H}\widetilde{R}(t-u)|t-u|^{H},
\end{equation}
if $0<h_{k}<1$ and $h_{k} + h_{l} \neq 1 $ for $k,l \in \{1,\dots,p\}$, where 
\begin{equation*}
\widetilde{R}(t)=
\begin{cases}
\widetilde{R},		&\hspace*{0.2cm} \text{if }	t>0,\\
\widetilde{R}',		&\hspace*{0.2cm} \text{if }	t<0
\end{cases}
\end{equation*}
and $\widetilde{R}=(\widetilde{R}_{kl})_{k,l=1,\dots, p}$ is defined as
\begin{equation} \label{equality_tildeRij}
\widetilde{R}_{kl}=		
\operatorname{B}(h_{k}+\frac{1}{2}, h_{l}+\frac{1}{2})
\left( \widetilde{c}_{1,kl} 
\frac{\cos(\pi h_{k} )}{\sin(\pi (h_{k}+h_{l}))}+
\widetilde{c}_{2,kl}+\widetilde{c}_{3,kl} \frac{\cos(\pi h_{l})}{\sin(\pi (h_{k}+h_{l}))}		 \right)
\end{equation}
with
\begin{equation*}
\widetilde{c}_{1,kl}=(M^{-}(M^{-})')_{kl},
\hspace*{0.2cm}
\widetilde{c}_{2,kl}=(M^{-}(M^{+})')_{kl},
\hspace*{0.2cm}
\widetilde{c}_{3,kl}=(M^{+}(M^{+})')_{kl}.
\end{equation*}
\par
In connection to the sample autocovariance process, the limit process will possibly be non-Gaussian. We will represent it by means of double Wiener-It\^{o} integrals, as a matrix-valued generalization of the univariate Rosenblatt process.
For the sake of simplicity, we define it in a vectorized form using the $\vecop$ operator. The $\vecop$ operator transforms a matrix into a vector by stacking the columns of the matrix one underneath the other.
Let $L^{2}(\RR^2, \RR^{p^2 \times p^2})$ denote the space of all functions $f:\RR^2 \to \RR^{p^2 \times p^2}$ equipped with the norm $\Vert f \Vert^2= \int_{\RR^{2}} \| f(x_{1},x_{2}) \|^2_{F} dx_{1}dx_{2}<\infty$. Then, for $f \in L^{2}(\RR^2, \RR^{p^2 \times p^2})$, we define a double Wiener-It\^{o} integral with respect to a multivariate real-valued Gaussian random measure $W(dx)$ as
\begin{equation} \label{eq:doublestochasticintegral}
I_{2}(f)=
\int_{\RR^2}' f(x_{1},x_{2}) \vecop\Big(W(dx_{1})W'(dx_{2}) \Big),
\end{equation}
where $\int_{\RR^2}'$ means that integration excludes the diagonals. See, for example \cite{Major2014}, for more information about multiple Wiener-It\^{o} integrals. 
Then, the $\RR^{p^2}$-valued process $\{Z(t)\}_{t \in \RR}$ is defined as
\begin{equation} \label{equality_limit_process_non_Gaussian}
Z(t)=I_{2}(f_{H,t}),
\end{equation}
with $f_{H,t}: \RR^2 \to \RR^{p^2 \times p^2}$ given by $f_{H,t}(x_{1},x_{2}):=f_{H,t,M^{+},M^{-}}(x_{1},x_{2})$ with
\begin{equation}  \label{equation_nu_Rosenblatt}
f_{H,t,M^{+},M^{-}}(x_{1},x_{2}) = \sum_{s_{1},s_{2} \in \{+,-\}} \int_{0}^{t} ((v-x_{2})^{H-I_{p}}_{s_{2}} \otimes (v-x_{1})^{H-I_{p}}_{s_{1}}) (M^{s_{2}} \otimes  M^{s_{1}})  dv,
\end{equation}
where $\otimes$ denotes the Kronecker product and $M^{+}, M^{-} \in \RR^{p \times p}$. The eigenvalues of the symmetric matrix $H \in \RR^{p \times p}$ are assumed to satisfy $h_{k} + h_{l} \in (0,1/2)$ for $k,l=1,\dots, p$.
Like the OFBM, the process $\{Z(t)\}_{t \in \RR}$ is supposed to be proper.
It is also operator self-similar and has stationary increments; see Lemma \ref{Lemma_properties_Rosen}. More precisely, the process $\{Z(t)\}_{t \in \RR}$ is operator self-similar with
scaling family $\{\Delta_{c}:\RR^{p^2} \to \RR^{p^2} ~|~ c>0\}$, where
$\Delta_{c} = c^{H} \otimes c^{H}$ that means $Z(ct) \overset{f.d.d.}{=} \Delta_{c} Z(t)$.

\section{Convergence of the sample mean process} \label{section_3}
In this section, we state the convergence results for the vector-valued sample mean process.
The following theorem gives the asymptotic normality for a large class of multivariate linear processes and will serve as a helpful tool to investigate the functional limit theorems under different assumptions on the dependence structure.

\begin{theorem}	\label{Theorem_general_multivariate_linear_process}
Let $\{X_{n}\}_{n \in \ZZ}$ be a stationary linear process \eqref{equality_general_linear_process}
with $\sum_{j\in \ZZ} \| \VPsi_{j} \|_{F}^{2} < \infty$ and set $\widetilde{\Sigma}_{N}^2:=\E(S_{N}S_{N}')$.
Suppose there is a nonsingular matrix $\Sigma_{N}^2$ such that $\widetilde{\Sigma}_{N}^2\sim\Sigma_{N}^2$ componentwise, as $N \to \infty$, and
\begin{equation*}
\lim_{N\to \infty} \Var(\Sigma_{N}^{-1} S_{N}) =I_{p}.
\end{equation*}
If each diagonal entry of the matrix
$\Sigma_{N}^2=(\sigma^{2}_{kl}(N))_{k,l=1,\dots, p}$
goes to infinity as $N \to \infty$, then
\begin{equation*}
\Sigma_{N}^{-1} \sum_{n=1}^{N} X_{n} \overset{\mathcal{L}}{\longrightarrow} \mathcal{N}(0,I_{p}).
\end{equation*}
\end{theorem}

\begin{remark}
The conditions in Theorem \ref{Theorem_general_multivariate_linear_process} are satisfied under multivariate long- as well as under multivariate short-range dependence. The assumptions on the matrices $A^{+}$, $A^{-}$ in (L) and $\sum_{j \in \ZZ} (\psi_{kl,j})_{k=p_{1}+1,\dots,p;l=1,\dots,p}$ in (S) to have full rank ensure that there is a nonsingular matrix $\Sigma_{N}^2$ with $\widetilde{\Sigma}_{N}^2\sim\Sigma_{N}^2$ and $\lim_{N\to \infty} \Var(\Sigma_{N}^{-1} S_{N}) =I_{p}$. The matrix $\Sigma_{N}^2$ satisfies $\sigma^{2}_{kk}(N) \to \infty$ as $N \to \infty$.
See Lemma \ref{Lemma_cross_cov_mixture_sample_mean} for $t=u=1$.
\end{remark}

The following result is the functional central limit theorem for the sample mean process, allowing the multivariate linear process to admit either short- or long-range dependence.
The limit process in the result is Gaussian and given by
\begin{equation}	\label{equation_limit_process_G_sample_mean_process}
\mathcal{G}(t)=
\begin{pmatrix}
C_{H}^{-1} \mathcal{B}^{(p_{1})}_{H}(t) \\
\mathcal{W}^{(p_{2})}(t)
\end{pmatrix}
\end{equation}
with $C_{H}=H-\frac{1}{2}I_{p_{1}}$ and $H=D_{p_{1}}+\frac{1}{2}I_{p_{1}}$, where 
$\{\mathcal{B}^{(p_{1})}_{H}(t)\}_{ t \in [0,1] }$ is an $\RR^{p_{1}}$-valued OFBM restricted to the unit interval and $\{\mathcal{W}^{(p_{2})}(t)\}_{t \in [0,1]}$ is an 
$\RR^{p_{2}}$-valued multivariate Brownian motion 
with $\mathcal{W}^{(p_{2})}(t)=(\mathcal{W}_{p_{1}+1}(t),\dots, \mathcal{W}_{p}(t))'$. 
The cross-covariances of $\mathcal{B}^{(p_{1})}_{H}(t)$ are given in \eqref{equality_cov_OFBMS}. The corresponding matrix $\widetilde{R}$ is defined in \eqref{equality_tildeRij} and depends on the parameters $\widetilde{c}_{i,kl}$, $i=1,2,3$, which are defined in terms of the matrices $A^{+},A^{-}$ arising in (L).
The cross-covariances of $\{\mathcal{W}^{(p_{2})}(t)\}_{t \in [0,1]}$ can be written as
\begin{equation*}
\E \mathcal{W}^{(p_{2})}(t)\mathcal{W}^{(p_{2})'}(u) =
\min(t,u) \sum_{\ell \in \ZZ} \Gamma_{p_{2},\ell},
\hspace*{0.2cm} t,u \in [0,1].
\end{equation*}
The cross-covariance structure between $\mathcal{B}^{(p_{1})}_{k,H}(t)$ and $\mathcal{W}^{(p_{2})}_{l}(t)$ for $k = 1,\dots,p_{1}$ and $l = p_{1}+1,\dots,p $ is given by
\begin{equation*}
\E \mathcal{B}^{(p_{1})}_{k,H}(t)\mathcal{W}^{(p_{2})}_{l}(u) =		
\sum_{l=1}^{p} \widetilde{h}^{-1}_{k}
\Big((
 t^{\widetilde{h}_{k}} \alpha^{-}_{kl} +
u^{\widetilde{h}_{k}} \alpha^{+}_{kl}-
|t-u|^{\widetilde{h}_{k}} \alpha_{kl}(t-u))
\sum_{j \in \ZZ} \psi_{kl,j} \Big),
\end{equation*}
for $d_{k}+d_{l} \geq 0$, where $\alpha_{kl}(t)=\alpha_{kl}^{+} \mathds{1}_{\{t>0\}}+ \alpha_{kl}^{-}\mathds{1}_{\{t<0\}}$ and 
$\widetilde{h}_{k}=d_{k}+1$. Otherwise, when $d_{k}+d_{l} < 0$, the components are uncorrelated. Whenever we refer to the process $\{\mathcal{G}(t)\}_{t \in [0,1]}$ in
\eqref{equation_limit_process_G_sample_mean_process}, we mean the process with the previously described cross-covariance structure.
\begin{theorem}	\label{Theorem_invarianceprinciple}
Let $\{X_{n}\}_{n \in \ZZ}$ be a stationary linear process 
\eqref{equality_general_linear_process} whose components satisfy (L) and (S), with
$\E\| \varepsilon_{0} \|^{2+\delta}<\infty$ for some $\delta>0$.
Then,
\begin{equation*}
A_{N}^{-1}(H)S_{\lfloor Nt \rfloor} 
\overset{\mathcal{L}}{\longrightarrow} 
\mathcal{G}(t),
\hspace*{0.2cm} t \in [0,1],
\end{equation*}
in $D[0,1]^{p}$, where $\{\mathcal{G}(t)\}_{t \in [0,1]}$ is a Gaussian process given in
\eqref{equation_limit_process_G_sample_mean_process}.
The normalization 
$A_{N}(H)=\operatorname{diag}(N^{H},N^{\frac{1}{2}}I_{p_{2}})$
is such that there is a non-singular matrix $C(H) \in \RR^{p \times p}$ with
\begin{equation*}
\lim_{N \to \infty} \Var(A_{N}^{-1}(H) \sum_{n=1}^{N} X_{n} )=C(H).
\end{equation*}
\end{theorem}

\begin{remark} \label{re:dep_parameters}
The dependence parameters $d_{l}$ for $l \in \{p_{1}+1,\dots, p\}$ determine the short-range dependent components, 
while the dependence parameters $d_{k}$ with $k \in \{1,\dots, p_{1}\}$ determine the long-range dependent components.
Choosing $d_{l}$ with $l \in \{p_{1}+1,\dots, p\}$ small enough to get $d_{k}+d_{l}<0$ for $d_{k}$ with $k \in \{1,\dots, p_{1}\}$,
yields an asymptotic independence between the short- and long-range dependent components.
\end{remark}

\begin{remark} \label{remark_comp_smp}
As noted in the previous remark, the asymptotic independence in Theorem \ref{Theorem_invarianceprinciple} depends on the interplay between the dependence parameters of the long- and short-range dependent components.
In contrast, \cite{BaiTaqqu2013} proved that the short- and long-range dependent components of a vector of functions of a univariate long-range dependent process are always asymptotically independent; see Theorem 5 in \cite{BaiTaqqu2013}. 
By expressing the sample mean process in \eqref{equality_partial_sum_process} componentwise, each component can be viewed as the sample mean process of the sum of different linear processes.
An implication of Theorem 3.5 in \cite{BaiTaqqu} is that a vector of univariate linear processes which are either short- and long-range dependent converges to a vector whose components are either a univariate Brownian motion or fractional Brownian motion, which leads to non-proper limiting process.
In contrast, our assumptions (L) and (S) on the linear process \eqref{equality_general_linear_process} ensure a proper limiting process.
\end{remark}

\section{Convergence of the sample autocovariance process} \label{section_4}
We first introduce some notation. For simplicity, we write
$Y_{N,\ell}(t)=(\widehat{\Gamma}_{N,\ell}-\Gamma_{\ell})(t)$ and
denote the $kl$-th component of $Y_{N,\ell}(t)$ by
\begin{equation*}
Y_{kl,N,\ell}(t)=	
\frac{1}{N}\sum_{n=1}^{\lfloor Nt \rfloor} 
\left( X_{k,n} X_{l,n+\ell}-\E(X_{k,0} X_{l,\ell}) \right).
\end{equation*}
When $k=l$, it is well-known (see \citet[Theorem 3.3]{HorvathKokoszka}) that the normalized limit of 
$Y_{kk,N,\ell}(t)$ is the Rosenblatt process when $2d_{k} \in (\frac{1}{2},1)$ and the usual Brownian motion when $2d_{k} \in (0,\frac{1}{2})$. This suggests to consider the index sets
\begin{equation*}
\begin{aligned}
I_{\operatorname{L}}&=\{(k,l) \in \{1,\dots,p_{1}\}^2 | d_{k}+d_{l} \in (\frac{1}{2},1) \},\\
\hspace*{0.2cm}
I_{\operatorname{S}}&=\{(k,l) \in \{1,\dots,p\}^2 | d_{k}+d_{l} \in (-\infty,\frac{1}{2}) \},
\end{aligned}
\end{equation*}
where the subscripts L and S refer to the long- and short-range dependence of the expected limits, that is the Rosenblatt process and Brownian motion, respectively.
\par
For a matrix $M=(M_{kl})_{k,l=1,\dots, p}$ and an index set $I\subset \{1,\dots, p\}^2$, we set
\begin{equation*}
\vecop_{I}(M)=E_{I}\vecop(M),
\end{equation*}
where $E_{I} \in \{0,1\}^{|I| \times p^2}$ denotes an elimination matrix, which transforms $\vecop(M)$ into a vector including only the matrix elements with indices in $I$.
Then, $Y_{N,\ell}(t)$ is partitioned into
\begin{equation*}
Y_{N,\ell}^{\operatorname{L}}(t)=\vecop_{I_{\operatorname{L}}}(Y_{N,\ell}(t)),
\hspace{0.2cm}
Y_{N,\ell}^{\operatorname{S}}(t)=\vecop_{I_{\operatorname{S}}}(Y_{N,\ell}(t))
\end{equation*}	
and $\widehat{\Gamma}_{N,\ell}(t)$ given in \eqref{equality_Gammahat} into
\begin{equation} \label{eq:sampleSL}
\widehat{\Gamma}^{\operatorname{L}}_{N,\ell}(t)=		
\vecop_{I_{\operatorname{L}}}(\widehat{\Gamma}_{N,\ell}(t)),
\hspace{0.2cm}
\widehat{\Gamma}^{\operatorname{S}}_{N,\ell}(t)=	
\vecop_{I_{\operatorname{S}}}(\widehat{\Gamma}_{N,\ell}(t)).
\end{equation}			
The limit process $\{Z^{\operatorname{L}}(t)\}_{t \in [0,1]}$ of $\widehat{\Gamma}^{\operatorname{L}}_{N,\ell}(t)$ in the result given below is defined as
\begin{equation}	\label{equation_ZL_Z_reduced_to_L}
Z^{\operatorname{L}}(t)=E_{I_{\operatorname{L}}} Z(t),
\end{equation}
where $Z(t)$ is given in \eqref{equality_limit_process_non_Gaussian}. The corresponding function is defined in \eqref{equation_nu_Rosenblatt} with $H=D$, $M^{+}=((A^{+})' ~~ 0_{p \times p_{2}})'$ and $M^{-}=((A^{-})' ~~ 0_{p \times p_{2}})'$, where $A^{+},A^{-}$ are given in (L)
and $0_{p \times p_{2}}$ denotes an $p \times p_{2}$ matrix with all entries equal to zero.
The limit process $\{G^{\operatorname{S}}_{\ell}(t)\}_{t \in [0,1]}$ of $\widehat{\Gamma}^{\operatorname{S}}_{N,\ell}(t)$ is a multivariate Brownian motion with cross-covariances
\begin{equation} \label{equality_covariance_Gaussian_Ishort}
\begin{aligned}
\Cov(G^{\operatorname{S}}_{\ell_{1}}(t),G^{\operatorname{S}}_{\ell_{2}}(u))=	&	
\min(t,u)E_{I_{\operatorname{S}}} \Bigg(  \sum_{r \in \ZZ} \Big( 
\Gamma_{r+\ell_{2}} \otimes \Gamma_{r-\ell_{1}}+
K_{p} (	\Gamma_{r} \otimes \Gamma_{r+\ell_{2}-\ell_{1}}	) 
\Big)  \\
& \hspace{1cm} 
+  \sum_{r\in \ZZ} \sum_{i\in \ZZ} 
(\VPsi_{i+\ell_{1}} \otimes \VPsi_{i}) \Sigma
(\VPsi_{i+r} \otimes \VPsi_{i+r+\ell_{2}} )'	 \Bigg) E_{I_{S}}',
\end{aligned}	 
\end{equation}
where 
\begin{equation}	\label{equation_big_Sigma}
\Sigma:=\sigma^{*} -\vecop(I_{p})	(\vecop(I_{p})	)'	-I_{p^2}-K_{p}
\end{equation}
with $ \sigma^{*} := \E(\vecop(\varepsilon_{0}\varepsilon_{0}')
(\vecop(\varepsilon_{0}\varepsilon_{0}'))')$. 
Furthermore, $K_{p}$ denotes the commutation matrix, which transforms $\vecop(M)$ into $\vecop(M')$ for a matrix $M=(M_{kl})_{k,l=1,\dots,p}$; see \cite{MagnusNeudecker} for more details on these kind of operations.
\par
In order to characterize the joint distribution of the processes $\{Z^{\operatorname{L}}(t)\}_{t \in [0,1]}$ in \eqref{equation_ZL_Z_reduced_to_L} and $\{G^{\operatorname{S}}_{\ell}(t)\}_{t \in [0,1]}$ in \eqref{equality_covariance_Gaussian_Ishort} we give the cross-covariance structure between the two Gaussian processes $W(t)$ and $G^{\operatorname{S}}_{\ell}(t)$, where $W(t)$ induces the random measure in the integral representation \eqref{equation_ZL_Z_reduced_to_L}.
The cross-covariance structure is given by
\begin{equation*}
\begin{aligned}
\E(G^{\operatorname{S}}_{\ell}(t)W'(u))
=	
\min\{t,u\} \sum_{ r \in \ZZ} \sum_{i \in \ZZ} E_{I_{\operatorname{S}}} 
(\VPsi_{r+\ell-i} \otimes \VPsi_{r-i})\widetilde{\Sigma},
\end{aligned}
\end{equation*}
where $\widetilde{\Sigma}=\E(\vecop(\varepsilon_{0}\varepsilon_{0} ')\varepsilon_{0}' )$. See also Remark 4.3. for more information about the dependence structure.
The following theorem gives the joint convergence of $\widehat{\Gamma}^{\operatorname{L}}_{N,\ell}(t)$ and $\widehat{\Gamma}^{\operatorname{S}}_{N,\ell}(t)$.

\begin{theorem}	\label{Theorem_Mixture_SRD_LRD}
Let $\{X_{n}\}_{n \in \ZZ}$ be a stationary linear process 
\eqref{equality_general_linear_process} whose components satisfy (L) and (S), with $E\| \varepsilon_{0} \|^{5} < \infty$.Then,
\begin{equation*}
\Bigg(
B_{N}^{-1}(D)
\begin{pmatrix}
Y_{N,\ell}^{\operatorname{L}}(t) \\
Y_{N,\ell}^{\operatorname{S}}(t) 
\end{pmatrix}
, \ell=0,\dots,L \Bigg)
\overset{\mathcal{L}}{\longrightarrow}
\Bigg(
\begin{pmatrix}
Z^{\operatorname{L}}(t) \\
G^{S}_{\ell}(t)
\end{pmatrix}
,\ell=0,\dots,L\Bigg),
\hspace*{0.2cm}
t \in [0,1],
\end{equation*}
in $D[0,1]^{p^2}$, where $\{Z^{\operatorname{L}}(t)\}_{t \in [0,1]}$ is defined in \eqref{equation_ZL_Z_reduced_to_L}, 
$\{G^{\operatorname{S}}_{\ell}(t)\}_{t \in [0,1]}$ in \eqref{equality_covariance_Gaussian_Ishort}. Furthermore, $Z^{\operatorname{L}}(t)$ and $G^{\operatorname{S}}_{\ell}(t)$ are uncorrelated but not independent.
The normalization $B_{N}(D)=\diag(\Delta_{N},N^{\frac{1}{2}}I_{|I_{\operatorname{S}}|})$
with $\Delta_{N} = E_{I_{\operatorname{L}}}( N^{D-\frac{1}{2}I_{p}} \otimes N^{D-\frac{1}{2}I_{p}} )E_{I_{\operatorname{L}}}'$ is such that there are non-singular matrices 
$C_{1}(D_{p}), C_{2}$ with
\begin{equation*}
\lim_{N \to \infty} \Cov(\Delta_{N}^{-1}\widehat{\Gamma}^{\operatorname{L}}_{N,\ell}(1),
\Delta_{N}^{-1}\widehat{\Gamma}^{\operatorname{L}}_{N,\ell}(1))=C_{1}(D_{p})
\hspace{0.2cm}\text{ and }\hspace{0.2cm}
\lim_{N \to \infty} N\Cov(
\widehat{\Gamma}^{\operatorname{S}}_{N,\ell}(1),
\widehat{\Gamma}^{\operatorname{S}}_{N,\ell}(1))
=	C_{2}. 
\end{equation*} 
\end{theorem}

\begin{remark} \label{Remark3}
The sum of two dependence parameters $d_{k}+d_{l}$ determines if the corresponding component of the sample autcovariance process is long- or short-range dependent. The case when a component behaves long-range dependent is characterized by $d_{k}+d_{l}\in (\frac{1}{2},1)$ and can only occur when the sample autocovariances between two long-range dependent components are considered. 
\end{remark}

\begin{remark} \label{Remark4}
As stated in Theorem \ref{Theorem_Mixture_SRD_LRD}, the processes $\{Z^{\operatorname{L}}(t)\}_{t \in [0,1]}$ and $\{G^{\operatorname{S}}_{\ell}(t)\}_{t \in [0,1]}$ are uncorrelated but not independent. 
To understand why these processes are not independent, note that the sample autocovariances in \eqref{equality_autocovariance_process} can be separated into diagonal and off-diagonal parts; see \eqref{eq:decomp_off_diagonal}.
While the diagonal terms are asymptotically negligible for the the long-range dependent components (see Lemma \ref{Lemma_convergence_diagonal}), the diagonals are crucial for the asymptotic behavior of the short-range dependent components. According to \eqref{eq:correlated}, the diagonals in the short-range dependent components influence the resulting dependence structure and lead to dependent limiting processes $\{Z^{\operatorname{L}}(t)\}_{t \in [0,1]}$ and $\{G^{\operatorname{S}}_{\ell}(t)\}_{t \in [0,1]}$.
\end{remark}

\begin{remark} \label{remark_comp2}
In contrast to Remark \ref{Remark4}, \cite{BaiTaqqu} studied vectors of univariate multilinear polynomial form processes whose filters depend on either short- or long-range dependent components and exclude the diagonals. The resulting limit theorem gives asymptotic independence between the short- and long-range dependent components when the linear forms are at least of order two; see Theorem 3.5 in \cite{BaiTaqqu}. 
Our setting can also be compared to \cite{BaiTaqqu2013} when the Hermite rank in each component is supposed to be two, which is the same as considering a vector of univariate sample autocovariances. In this case, the short- and long-range dependent components are asymptotically independent; see Theorem 5 in \cite{BaiTaqqu2013}. 
Furthermore, our setting ensures a proper limiting process, while the limits in \cite{BaiTaqqu, BaiTaqqu2013} are not necessarily proper.
\end{remark}

\section{Proofs}

\subsection{Proof of Theorem \ref{Theorem_general_multivariate_linear_process}}
We first state a lemma from \cite{Rack_Suquet_Op_2011}, which gives sufficient conditions for two linear processes with values in an arbitrary Hilbert space to have the same convergence behavior. 
Let $\Hi$ and $\mathbb{E}$ be two Hilbert spaces and $\{\varepsilon_{j}\}_{ j \in \ZZ }$ a sequence of i.i.d.\ random variables with values in $\mathbb{E}$.
Define $\{X_{n}\}_{ n \in \ZZ }$ with $X_{n}=\sum_{j\in \ZZ} D_{nj} \varepsilon_{j}$ and $D_{nj} \in L(\Hi,\mathbb{E})$, the space of bounded linear operators from $\Hi$ to $\mathbb{E}$. Similarly, define $\{Y_{n}\}_{ n \in \ZZ }$ with $Y_{n}=\sum_{j\in \ZZ} D_{nj} \widetilde{\varepsilon}_{j}$,
where $D_{nj}$ is the same operator as in $X_{n}$ and $\widetilde{\varepsilon}_{j}$ is a sequence of Gaussian random elements with values in $\mathbb{E}$, zero mean and the same covariance operator as $\varepsilon_{j}$.
The notation $\| \cdot \|_{op}$ stands for the operator norm.
\begin{lemma} \label{Lemma_Suquet_Rack}
If
\begin{equation} \label{conditions_Suquet_Rack}
\lim_{n \to \infty} \sup_{j \in \ZZ} \Vert D_{nj} \Vert_{op}=0 
\hspace*{0.2cm} \text{ and } \hspace*{0.2cm}
\limsup_{n \to \infty} \sum_{j \in \ZZ} \Vert D_{nj} \Vert_{op}^2<\infty,
\end{equation}
then
\begin{equation*}
\lim_{n \to \infty} \varrho_{3} (X_{n},Y_{n})=0,
\end{equation*}
where the metric $\varrho_{k}$ is defined by 
\begin{equation*}
\varrho_{k} (X,Y)=\sup_{f \in F_{k}} \left| \E f(X) - \E f(Y) \right|,
\end{equation*}
for the set $F_{k}$ of all $k$ times Fr\`{e}chet differentiable functions $f:\Hi\rightarrow \RR$ such that\\
$\sup_{x \in \Hi} \vert f^{(i)}(x) \vert \leq 1$ for $i\in \{0,\dots,k\}$.
\end{lemma}
The proof of Lemma \ref{Lemma_Suquet_Rack} is given in \cite{Rack_Suquet_Op_2011}.
The processes $X_{n}$ and $Y_{n}$ have the same convergence behavior if $\lim_{n \to \infty} \varrho_{3} (X_{n},Y_{n})=0$, since the metric induces the weak topology on the set of probability measures on $\Hi$; see \cite{Gine_Leon_On_the_central_limit_theorem_in_Hilbert_space}.
\par
We next rewrite the normalized sample mean $\Sigma^{-1}_{N} \sum_{n=1}^{N} X_{n} $ as a linear process and prove for it the conditions \eqref{conditions_Suquet_Rack}.
Thus, let $\Sigma_{N}^{-1}=: (\Sigma_{kl,N})_{k,l=1,\dots,p}$, which exists since $\Var(a'S_{N}) \neq 0$ for 
$a \in \RR \backslash \{0\}$ by assumption.
For $\lambda= (\lambda_{1},\dots,\lambda_{p}) \in \RR^{p} $, write
\begin{equation} \label{equality_sequence_of_matrices}
\begin{aligned}
\lambda'\Sigma^{-1}_{N} \sum_{n=1}^{N} X_{n}
=	&	
\sum_{j\in \ZZ} \sum_{n=1}^{N} \sum_{k=1}^{p}
\sum_{i=1}^{p} \lambda_{i} \Sigma_{ik,N} \sum_{l=1}^{p}  
\psi_{kl,n-j} \varepsilon_{l,j}	\\
=:	&	
\sum_{l=1}^{p} \sum_{j\in \ZZ} B_{l,Nj} \varepsilon_{l,j}
=	\phantom{:}
\sum_{j\in \ZZ} \left( B_{1,Nj},\dots,B_{p,Nj} \right) 
\begin{pmatrix}
\varepsilon_{1,j}\\
\vdots\\
\varepsilon_{p,j}
\end{pmatrix}
=: 
\sum_{j \in \ZZ} B_{Nj}' \varepsilon_{j},
\end{aligned}
\end{equation}
which is the form needed to apply Lemma \ref{Lemma_Suquet_Rack}.
Since $B_{Nj} \in \RR^{p}$, the operator norm is $\Vert B_{Nj} \Vert_{op}=\max_{1 \leq l \leq p} \vert B_{l,Nj} \vert$.
To show that $B_{Nj}$ satisfies the conditions \eqref{conditions_Suquet_Rack}, we need the following auxiliary lemma concerning the variances of one entry of the sample mean,
\begin{equation}	\label{Definition_small_omega}
\omega^{2}_{kl}(N):=	
\E(\sum_{n=1}^{N} \sum_{j\in \ZZ} 
\psi_{kl, n-j}\varepsilon_{l,j})^2=	
\sum_{j\in \ZZ} (\sum_{n=1}^{N} \psi_{kl, n-j})^2.
\end{equation}

\begin{lemma}	\label{Lemma_convergence_to_zero}
The sequence of matrix entries 
$(\Sigma_{ik,N} \omega^{\frac{1}{2}}_{kl}(N))_{N \geq 1}$ 
converges to zero for each $k,l,i \in \{1,\dots,p\}$.
\end{lemma}
\begin{proof}
Define the matrix $\Omega^2_{N}=(\Omega_{kl,N})_{k,l=1,\dots,p}$ by
\begin{equation*}
\Omega_{kl,N}=
\begin{cases}
1,			&\text{if } k=l,	\\
\frac{\sigma_{kl}^{2}(N)}{\sigma_{kk}(N)\sigma_{ll}(N)}, &\text{if } k \neq l,
\end{cases}	
\end{equation*}
so that $\Sigma_{N}^{2}=\operatorname{diag}(\sigma_{11}(N),\dots,\sigma_{pp}(N)) \Omega^2_{N} \operatorname{diag}(\sigma_{11}(N),\dots,\sigma_{pp}(N))$. Then, there is a matrix 
$C \in \RR^{p \times p}$, whose diagonal entries are equal to one and the off-diagonal elements are constants $c_{ij}$ such that $\lim _{N \to \infty} \Omega_{kl,N}=C$ componentwise. This implies
\begin{equation*}
\begin{aligned}
\lim_{N \to \infty}\Sigma_{N}^{-2}& =		
\lim_{N \to \infty}(\Sigma_{N}^{2})^{-1}	\\& =		
\lim_{N \to \infty} (\operatorname{diag}(\sigma_{11}(N),\dots,\sigma_{pp}(N)) \Omega^2_{N} 
\operatorname{diag}(\sigma_{11}(N),\dots,\sigma_{pp}(N)))^{-1} \\& =		
\lim_{N \to \infty} \operatorname{diag}\left( \frac{1}{\sigma_{11}(N)},\dots,\frac{1}{\sigma_{pp}(N)}\right) C^{-1} 
\operatorname{diag}\left( \frac{1}{\sigma_{11}(N)},\dots,\frac{1}{\sigma_{pp}(N)}\right),
\end{aligned}
\end{equation*}
since matrix inversion is a continuous transformation. This leads to
\begin{equation*}
\begin{aligned}
&		 	
\lim_{N \to \infty} (\Sigma_{ik,N}  \omega^{\frac{1}{2}}_{kl}(N))_{i,l=1,\dots,p}   
\\& =	
\big(\lim_{N \to \infty} \diag(\omega_{1l}(N),\dots,\omega_{pl}(N)) 
\Sigma_{N}^{-2}  \big)^{\frac{1}{2}} 	 
\\& =
\left(\lim_{N \to \infty} \operatorname{diag}\left( 
\frac{\omega_{1l}(N)}{\sigma_{11}(N)},\dots,\frac{\omega_{pl}(N)}
{\sigma_{pp}(N)}\right)  
C^{-1} \operatorname{diag}\left( \frac{1}{\sigma_{11}(N)},\dots,\frac{1}{\sigma_{pp}(N)}\right) \right) ^{\frac{1}{2}}	  
\\& =
\lim_{N \to \infty} \left( \frac{\omega_{kl}(N)\widetilde{c}_{ik}}
{\sigma_{kk}(N)\sigma_{ii}(N)} \right)_{i,l=1,\dots, p} 
\end{aligned}
\end{equation*}
with $C^{-1}=(\widetilde{c}_{ij})_{i,j=1,\dots, p}$ and finally
\begin{equation*}
\lim_{N \to \infty} \frac{ \omega_{kl}(N) }{\sigma_{kk}(N)\sigma_{ii}(N)}
=	
\lim_{N \to \infty} \frac{(\sum_{j\in \ZZ} (\sum_{n=1}^{N} \psi_{kl,n-j})^2)^{\frac{1}{2}}}{\sigma_{kk}(N)\sigma_{ii}(N)}
=		0
\end{equation*}
with $\sigma^2_{kk}(N) \sim \sum_{m=1}^{p}  \sum_{j\in \ZZ} ( \sum_{n=1}^{N} \psi_{km,n-j})^2$.
\end{proof}	
		
\begin{lemma} \label{Lemma_conditions_satisifed}
The sequence of matrices $(B_{Nj}')_{N \geq 1,j \in \ZZ}$ in \eqref{equality_sequence_of_matrices} satisfies the conditions \eqref{conditions_Suquet_Rack}.
\end{lemma}
\begin{proof}
To prove the conditions \eqref{conditions_Suquet_Rack}, we consider 
$\vert B_{l,Nj} \vert$ for each $l \in \{1,\dots, p\}$ instead of 
$\max_{1 \leq l \leq p} \vert B_{l,Nj} \vert$. Then,
\begin{equation*}
\sup_{j \in \ZZ} \vert B_{l,Nj} \vert
=			
\sup_{j \in \ZZ} \vert \sum_{n=1}^{N}  \sum_{i=1}^{p} \sum_{k=1}^{p} \lambda_{i} \Sigma_{ik,N}
 \psi_{kl,n-j} \vert	
\leq		
\sup_{j \in \ZZ} \sum_{i=1}^{p} \sum_{k=1}^{p} |\lambda_{i}| 
|\Sigma_{ik,N} \omega^{\frac{1}{2}}_{kl}(N) |
\frac{| \sum_{n=1}^{N} \psi_{kl,n-j} |}{\omega^{\frac{1}{2}}_{kl}(N)}.
\end{equation*}		
By using the inequality 
\begin{equation}	\label{inequality_omega}
\frac{| \sum_{n=1}^{N} \psi_{kl,n-j} |^2}{\omega^{2}_{kl}(N)}
\leq		
\frac{4}{\omega_{kl}(N)}
\left( \frac{ \sum_{j \in \ZZ} \psi^2_{kl,j}}{\omega_{kl}(N)}
+ \Big( \sum_{j \in \ZZ} \psi_{kl,j}^2\Big)^{\frac{1}{2}} \right)
\end{equation}
from \citet[Theorem 18.6.5.]{IbragimovLinnik}, we further get that
\begin{equation*}
\sup_{j \in \ZZ} \vert B_{l,Nj} \vert
\leq		
2 \sum_{i=1}^{p} \sum_{k=1}^{p} |\lambda_{i}| 
|\Sigma_{ik,N} \omega^{\frac{1}{2}}_{kl}(N) |
\left( \frac{ \sum_{j \in \ZZ} \psi_{kl,j}^2}{\omega_{kl}(N)}
+ \Big( \sum_{j \in \ZZ} \psi_{kl,j}^2\Big)^{\frac{1}{2}} \right) ^{\frac{1}{2}}
\to 		0,
\end{equation*}
as $N \to \infty$ since $\Sigma_{ik,N} \omega^{\frac{1}{2}}_{kl}(N) $ converges to zero for each 
$i,k,l \in \{1,\dots,p\}$ by Lemma \ref{Lemma_convergence_to_zero}.
This proves the first condition in \eqref{conditions_Suquet_Rack}.
The second condition holds, since
\begin{equation*}
\begin{aligned}
\limsup_{N \to \infty} \sum_{j \in \ZZ} |B_{l,Nj} |^2
& = 	
\limsup_{N \to \infty} \sum_{j \in \ZZ}
| \sum_{n=1}^{N}  \sum_{i=1}^{p} \lambda_{i} 
\sum_{k=1}^{p} \Sigma_{ik,N} \psi_{kl,n-j} |^2	\\	
& \leq  	
\limsup_{N \to \infty} \sum_{j \in \ZZ} \sum_{l=1}^{p}| 
\sum_{n=1}^{N}  \sum_{i=1}^{p} \lambda_{i} 
\sum_{k=1}^{p} \Sigma_{ik,N} \psi_{kl,n-j} |^2	
=	 		
\sum_{i=1}^{p} \lambda_{i}^2 < \infty,
\end{aligned}
\end{equation*}
where we used the fact that the variances of the normalized sample mean satisfy
\begin{equation}	\label{equality_variances}
\lambda' \lambda
=			
\lim_{N \to \infty} \E(\lambda' \Sigma_{N}^{-1} \sum_{n=1}^{N} X_{n})^2	
=			
\lim_{N \to \infty} \sum_{j \in \ZZ} \sum_{l=1}^{p}| 
\sum_{n=1}^{N}  \sum_{i=1}^{p} \lambda_{i} 
\sum_{k=1}^{p} \Sigma_{ik,N} \psi_{kl,n-j} |^2.
\end{equation}
\end{proof}
By Lemma \ref{Lemma_conditions_satisifed}, the variables $\lambda' \Sigma_{N}^{-1} \sum_{n=1}^{N} X_{n}$ behave like Gaussian. The variances are given by
\eqref{equality_variances}, so $\lambda' \Sigma_{N}^{-1} \sum_{n=1}^{N} X_{n}$ converges in distribution to $\lambda'Z$ where $Z$ follows the $\mathcal{N}(0,I_{p})$ distribution.

\subsection{Proof of Theorem \ref{Theorem_invarianceprinciple}}	
In order to prove Theorem \ref{Theorem_invarianceprinciple}, we first present an auxiliary result regarding the limit processes covariance structure (Section \ref{se:SampleMeanAr}). We then investigate the convergence of the finite-dimensional distributions and tightness in $D[0,1]^{p}$ (Sections \ref{se:SampleMeanCfdd} and \ref{se:SampleMeanTightness}), which establish Theorem \ref{Theorem_invarianceprinciple}.

\subsubsection{Auxiliary result} \label{se:SampleMeanAr}
We examine the asymptotic covariance structure in the following auxiliary result.
\begin{lemma} \label{Lemma_cross_cov_mixture_sample_mean}
Under the assumptions in Theorem \ref{Theorem_invarianceprinciple},
\begin{equation*}
\lim_{N \to \infty} \Cov(A^{-1} _{N}(H)S_{\lfloor Nt \rfloor}, A^{-1}_{N}(H)S_{\lfloor Nu \rfloor})
=			
\Cov(\mathcal{G}(t),\mathcal{G}(u)),
\end{equation*}
where $\{\mathcal{G}(t)\}_{t \in [0,1] }$ is defined in 
\eqref{equation_limit_process_G_sample_mean_process}.
Furthermore, $A_{N}(H)=\operatorname{diag}(N^{H},N^{\frac{1}{2}}I_{p_{2}})$.
\end{lemma}
\begin{proof}
The proof is divided into three parts: (i) we examine the covariances between the long-range dependent components, (ii) the covariances between the short-range dependent components and (iii), we consider the mixture terms. For each part, note that interchanging the order of summation and assuming $t<u$ leads to
\begin{equation}	\label{eqaution_covariance_sample_mean_general}
\begin{aligned}
& \E(\sum_{n=1}^{\lfloor Nt \rfloor} X_{k,n}
\sum_{n=1}^{\lfloor Nu \rfloor} X_{l,n})	
\\&=	
\lfloor Nt \rfloor \gamma_{kl}(0) +
\sum_{n=1}^{\lfloor Nt \rfloor} \left( \lfloor Nt \rfloor - n \right) 
(\gamma_{kl}(n)+	\gamma_{lk}(n)) +
\sum_{n=1}^{m_{N}-1}n\gamma_{kl}(n)	 
\\& \hspace{1cm}	
+ m_{N}  \sum_{n=m_{N}}^{\lfloor Nu \rfloor -m_{N}}\gamma_{kl}(n)+ 
\sum_{n=\lfloor Nu \rfloor-m_{N}+1}^{\lfloor Nu \rfloor-1}
\left( \lfloor Nu \rfloor - n \right)\gamma_{kl}(n),
\end{aligned}
\end{equation}	
where $m_{N}=\min(\lfloor Nt \rfloor,\lfloor Nu \rfloor-\lfloor Nt \rfloor)$.
\par
\textit{Part (i):} By Proposition \ref{Proposition_mLRD}, the underlying process $\{X^{\operatorname{L}}_{n}\}_{n \in \ZZ}$ satisfies \eqref{equality_Definition_KP_multivariate_stationary_lrd} with \eqref{equation_Rij}. The proof follows by applying 
\eqref{eqaution_covariance_sample_mean_general} and similar arguments as in the univariate case (see for example the proof of Proposition 2.8.8 in \cite{PipirasTaqqu}) to each component so that
\begin{equation*}
\begin{aligned}
&		
\lim_{N \to \infty} \E(N^{-H} \sum_{n=1}^{\lfloor Nt \rfloor} X^{\operatorname{L}}_{n}
(N^{-H} \sum_{n=1}^{\lfloor Nu \rfloor} X^{\operatorname{L}'}_{n}) )	\\
& =		
\lim_{N \to \infty} \left( N^{-(1+d_{k}+d_{l})}
\E(	\sum_{n=1}^{\lfloor Nt \rfloor} X_{k,n}
\sum_{n=1}^{\lfloor Nu \rfloor} X_{l,n})	\right) _{k,l=1,\dots,p_{1}}		\\
& =		
\Bigg( 		  
\frac{1}{(d_{k}+d_{l})(1+d_{k}+d_{l})}	\\
&  
\phantom{=\Bigg( }	
(R_{kl}t^{1+d_{k}+d_{l}}+R_{lk}u^{1+d_{k}+d_{l}}-
R_{kl}(t-u)|t-u|^{1+d_{k}+d_{l}})	\Bigg) _{k,l=1,\dots,p_{1}}	
\\& =			
C_{H}^{-1}(t^{H}\widetilde{R}t^{H}+u^{H}\widetilde{R}'u^{H}-|t-u|^{H}\widetilde{R}(t-u)|t-u|^{H}) C_{H}^{-1},
\end{aligned}
\end{equation*}
where $C_{H}=H-\frac{1}{2}I_{p_{1}}$ with $H=D_{p_{1}}+\frac{1}{2}I_{p_{1}}$, $R_{kl}$ is defined by \eqref{equation_Rij} and
\begin{equation*}
R_{kl}(t)=
\begin{cases}
R_{kl},	\hspace*{0.2cm} \text{if }	t>0,\\
R_{lk},	\hspace*{0.2cm} \text{if }	t<0.
\end{cases}
\end{equation*}
The matrix $\widetilde{R}$ is given in \eqref{equality_tildeRij} and depends on 
the parameters $\widetilde{c}_{i,kl}$, $i=1,2,3$, which are defined in terms of 
the matrices $A^{+},A^{-}$ arising in (L).
Using the basic properties of the beta and gamma function gives
\begin{equation*}
\frac{1}{(d_{k}+d_{l})(1+d_{k}+d_{l})} R_{kl} 
=	
\frac{1}{(h_{k}-\frac{1}{2})(h_{l}-\frac{1}{2})} \widetilde{R}_{kl},
\end{equation*}
since $h_{k}=d_{k}+\frac{1}{2}$.
\par
\textit{Part (ii):} By Proposition \ref{Proposition_mSRD}, the autocovariances of the process $\{X^{\operatorname{S}}_{n}\}_{n \in \ZZ}$ are absolutely summable. Then, the relation \eqref{eqaution_covariance_sample_mean_general} and standard arguments under univariate short-range dependence (see e.g. \citet[Proposition 3.3.1]{giraitis}) yield
\begin{align*}
&		
\lim_{N \to \infty} \E(N^{-\frac{1}{2}} \sum_{n=1}^{\lfloor Nu \rfloor} X^{\operatorname{S}}_{n}
(N^{-\frac{1}{2}} \sum_{n=1}^{\lfloor Nt \rfloor} X^{\operatorname{S}}_{n})' )	\\
& =	
\lim_{N \to \infty}N^{-1} \left( \lfloor Nt \rfloor \gamma_{kl}(0) +
\sum_{n=1}^{\lfloor Nt \rfloor} \left( \lfloor Nt \rfloor - n \right)	
(\gamma_{kl}(n)+	\gamma_{lk}(n))	\right)_{k=p_{1}+1,\dots,p; l=1,\dots,p_{1}}		\\
& =		
t \sum_{n \in \ZZ} \Gamma_{p_{2},n}
\end{align*}
for $t<u$, since $\Gamma_{p_{2},-n}=\Gamma_{p_{2},n}'$.
\par
\textit{Part (iii):}
For the covariances between the sample means of $\{X^{\operatorname{L}}_{n}\}_{n \in \ZZ}$ and $\{X^{\operatorname{S}}_{n}\}_{n \in \ZZ}$, we distinguish two cases: $d_{k}+d_{l}<0$ and 
$d_{k}+d_{l} \geq 0$, where $d_{k}$ are associated with the components satisfying (L) and $d_{l}$ with the components satisfying (S).
When $d_{k}+d_{l}<0$, the autocovariances are absolutely 
summable following the proof of Proposition \ref{Proposition_mSRD}, so that
\begin{equation*}
\begin{aligned}
N^{-(1+d_{k})} \sum_{n=1}^{\lfloor Nt \rfloor} \left( \lfloor Nt \rfloor - n \right)
\gamma_{kl}(n)
&=	
N^{-(1+d_{k})} \sum_{n=1}^{\lfloor Nt \rfloor} \left( \lfloor Nt \rfloor - n \right)
\sum_{j \in \ZZ } \sum_{m=1}^{p} \psi_{km,j} \psi_{lm,j+n} 	\\
&\sim	
t \sum_{m=1}^{p} N^{-d_{k}} \sum_{n=1}^{\infty} \sum_{j \in \ZZ} 
			\psi_{km,j} \psi_{lm,j+n} \to 0.
\end{aligned}
\end{equation*}
When $d_{k}+d_{l} \geq 0$, we consider only the summand
$\sum_{n=1}^{\lfloor Nt \rfloor} \left( \lfloor Nt \rfloor - n \right) \gamma_{kl}(n)$ in \eqref{eqaution_covariance_sample_mean_general} in detail since the others can be dealt with analogously.
Write the $kl$-th component of the autocovariance function as
\begin{equation} \label{eq:componentautocov}
\begin{aligned}
\gamma_{kl}(n)
&=\phantom{:}	
\sum_{m=1}^{p} \left( 
\sum_{j = -\infty}^{-n-1} \psi_{km,j} \psi_{lm,j+n} +
\sum_{j = -n}^{0} \psi_{km,j} \psi_{lm,j+n} +
\sum_{j =0}^{\infty} \psi_{km,j} \psi_{lm,j+n}	\right) \\
&=:	
\gamma_{1,kl}(n)+\gamma_{2,kl}(n)+\gamma_{3,kl}(n).
\end{aligned}
\end{equation}
Recall that $k \in \{1,\dots, p_{1}\}$ for the long-range dependent components and $l \in \{p_{1}+1,\dots, p\}$ for the short-range dependent components. For $\gamma_{1,kl}(n)$, we have
\begin{equation*}
\begin{aligned}
N^{-(1+d_{k})} \sum_{n=1}^{\lfloor Nt \rfloor} \left( \lfloor Nt \rfloor - n \right)
\gamma_{1,kl}(n)
&	=		
N^{-(1+d_{k})} \sum_{n=1}^{\lfloor Nt \rfloor} \left( \lfloor Nt \rfloor - n \right)
\sum_{j = -\infty}^{-n-1} \sum_{m=1}^{p}
\psi_{km,j} \psi_{lm,j+n} 	\\
&	=		
\sum_{m=1}^{p} N^{-(1+d_{k})} 
\sum_{n=1}^{\lfloor Nt \rfloor} \left( \lfloor Nt \rfloor - n \right)
\sum_{j = -\infty}^{-1} 
\psi_{km,j-n} \psi_{lm,j}	\\
& \sim	
\sum_{m=1}^{p} 
\alpha_{km}^{-} \frac{1}{d_{k}(d_{k}+1)}t^{d_{k}+1} \sum_{j = -\infty}^{-1} \psi_{lm,j}.
\end{aligned}			
\end{equation*}
For $\gamma_{2,kl}(n)$, we get
\begin{equation*}
\begin{aligned}
\gamma_{2,kl}(n)
&=		
\sum_{j = -n}^{0} \sum_{m=1}^{p}
\psi_{km,j} \psi_{lm,j+n} 	
=		
\sum_{m=1}^{p} 
\sum_{j = 0}^{n}
\psi_{km,j-n} \psi_{lm,j}	\\
&=		
\sum_{m=1}^{p} 
\sum_{j = 0}^{n} 
(\psi_{kl,j-n} -\psi_{km,-n}) \psi_{lm,j} 
+\sum_{m=1}^{p} \sum_{j = 0}^{n}\psi_{kl,-n} \psi_{lm,j}.
\end{aligned}
\end{equation*}
The last term in this expression determines the limit as
\begin{equation*}
\sum_{m=1}^{p} N^{-(1+d_{k})} 
\sum_{n=1}^{\lfloor Nt \rfloor} \left( \lfloor Nt \rfloor - n \right)
\psi_{km,-n} \sum_{j = 0}^{n} \psi_{lm,j} 
\sim		
\sum_{m=1}^{p} \alpha^{-}_{km} \frac{1}{d_{k}(d_{k}+1)} t^{d_{k}+1}
\sum_{j = 0}^{\infty} \psi_{lm,j},
\end{equation*}
whereas the other term is asymptotically negligible since
\begin{align*}
		&	|N^{-(1+d_{k})} 
			\sum_{n=1}^{\lfloor Nt \rfloor} \left( \lfloor Nt \rfloor - n \right)
			\sum_{j = 0}^{n} 
			(\psi_{km,j-n} -\psi_{lm,-n}) \psi_{lm,j}|	\\
& \sim		|N^{-(1+d_{k})} 
			\sum_{n=1}^{\lfloor Nt \rfloor} \left( \lfloor Nt \rfloor - n \right)
			\sum_{j = 0}^{n} 
			\alpha_{km}^{-}((n-j)^{d_{k}-1} -n^{d_{k}-1}) j^{d_{l}-1}C_{lm}(j)|	\\
& \leq		N^{-(1+d_{k})} 
			\sum_{n=1}^{\lfloor Nt \rfloor} \left( \lfloor Nt \rfloor - n \right)
			\sum_{j = 0}^{n} 
			|\alpha_{km}^{-}\beta |((n-j)^{d_{k}-1} -n^{d_{k}-1}) j^{d_{l}-1}	\\
& \sim		N^{-(1+d_{k})} 
			\sum_{n=1}^{\lfloor Nt \rfloor} \left( \lfloor Nt \rfloor - n \right) n^{d_{k}+d_{l}-1}
			|\alpha_{km}^{-}\beta | \int_{0}^{1} ((1-x)^{d_{k}-1} -1) x^{d_{l}-1} dx
\to		0,
\end{align*}
where the integral is finite since $d_{k}+d_{l} \geq 0$ (see \citet[p.\ 315]{gradshteyn2007}). For $\gamma_{3,kl}(n)$ note that
\begin{equation*}
			\gamma_{3,kl}(n)
	=		\sum_{m=1}^{p} 
			\sum_{j = 0}^{\infty}
			\psi_{km,j} \psi_{lm,j+n} 	
	=		\sum_{m=1}^{p} 
			\sum_{j = 0}^{\infty}
			C_{km}(j) j^{d_{k}-1} C_{lm}(j+n) (j+n)^{d_{l}-1}
\end{equation*}
and
\begin{equation*}
\begin{aligned}
			&		|\sum_{m=1}^{p} N^{-(1+d_{k})} 
					\sum_{n=1}^{\lfloor Nt \rfloor} (\lfloor Nt \rfloor-n)
					\sum_{j = 0}^{\infty}
					C_{km}(j) j^{d_{k}-1} C_{lm}(j+n) (j+n)^{d_{l}-1}|	\\
	&	\sim		
					|\sum_{m=1}^{p} N^{-(1+d_{k})} 
					\sum_{n=1}^{\lfloor Nt \rfloor} (\lfloor Nt \rfloor-n)
					\sum_{j = 0}^{\infty}
					\alpha_{km}^{+} j^{d_{k}-1} C_{lm}(j+n) (j+n)^{d_{l}-1}|		\\
	&	\leq		
					\sum_{m=1}^{p} N^{-(1+d_{k})} 
					\sum_{n=1}^{\lfloor Nt \rfloor} (\lfloor Nt \rfloor-n)
					|\alpha_{km}^{+} \beta | \sum_{j = 0}^{\infty}
					j^{d_{k}-1}  (j+n)^{d_{l}-1}	\\
	&	\sim 		
					\sum_{m=1}^{p} N^{-(1+d_{k})} 
					\sum_{n=1}^{\lfloor Nt \rfloor} (\lfloor Nt \rfloor-n)
					|\alpha_{km}^{+} \beta | n^{d_{k}+d_{l}-1} 
					\int_{0}^{\infty} x^{d_{k}-1}  (x+1)^{d_{l}-1} dx \to 		0 ,
\end{aligned}
\end{equation*}
where the integral is finite since $d_{k}+d_{l}<1$.
Combining the results for $\gamma_{1,kl}(n),\gamma_{2,kl}(n)$ and $\gamma_{3,kl}(n)$ yields
\begin{equation*}
N^{-(1+d_{k})} 
\sum_{n=1}^{\lfloor Nt \rfloor} \left( \lfloor Nt \rfloor - n \right) \gamma_{kl}(n)\sim
\sum_{m=1}^{p} \frac{1}{d_{k}(d_{k}+1)} t^{d_{k}+1} 
\alpha^{-}_{km} \sum_{j \in \ZZ} \psi_{lm,j}.
\end{equation*}
\par
Dealing similarly with the other summands in \eqref{eqaution_covariance_sample_mean_general} gives
\begin{equation*}
\begin{aligned}
&	
\lim_{N \to \infty} \E(N^{-H} \sum_{n=1}^{\lfloor Nt \rfloor} X^{\operatorname{L}}_{n}
N^{-\frac{1}{2}} \sum_{n=1}^{\lfloor Nu \rfloor} X^{\operatorname{S}'}_{n})		\\
&	=		
C_{H}^{-1} \widetilde{H}^{-1} (
t^{\widetilde{H}} A^{-} +
u^{\widetilde{H}} A^{+}-
|t-u|^{\widetilde{H}} A(t-u))
\left( \sum_{j \in \ZZ} \psi_{lm,j} \right)_{\substack{m = 1,\dots,p \\ l = p_{2},\dots,p }},	
\end{aligned}
\end{equation*}
where $C_{H}=H-\frac{1}{2}I_{p_{1}}$ with $H=D_{p_{1}}+\frac{1}{2}I_{p_{1}}$, $\widetilde{H}=D_{p_{1}}+I_{p_{1}}$ and 
\begin{equation*}
A(t)=
\begin{cases}
A^{+},	\hspace{0.2cm}	&\text{if } t>0,\\
A^{-}	,	\hspace{0.2cm}	&\text{if } t<0.
\end{cases}
\end{equation*}
\end{proof}

We conclude the section with a comment on the properness of the process $\{\mathcal{G}(t)\}_{t \in [0,1]}$, since it is a consequence of the previous Lemma \ref{Lemma_cross_cov_mixture_sample_mean}.
The matrices $A^{+}$, $A^{-}$ and $\Lambda:=\sum_{j \in \ZZ} (\psi_{kl,j})_{k=p_{1}+1,\dots,p;l=1,\dots,p}$ are assumed to have full rank by condition (L) and (S), respectively. For this reason, $A^{+}(A^{+})'$, $A^{-}(A^{-})'$ and 
$\Lambda \Lambda'$ are positive definite. This and the positive-semi definiteness of $\E \mathcal{G}(t)\mathcal{G}'(t)$ imply, that 
$\E \mathcal{G}(t) \mathcal{G}'(t)$ is positive definite. So, one can infer that $\{\mathcal{G}(t)\}_{t \in [0,1]}$ is proper.


\subsubsection{Convergence of the finite-dimensional distributions} \label{se:SampleMeanCfdd}
We prove the convergence of the finite-dimensional distributions by adapting the proof of Theorem \ref{Theorem_general_multivariate_linear_process}.
It is enough to verify that
\begin{equation} \label{eq:fddsamplemaen}
\lambda'A^{-1}_{N}(H) S_{\lfloor Nt \rfloor} 
\overset{f.d.d.}{\longrightarrow} 
\lambda' \mathcal{G}(t),
\end{equation}
where $\mathcal{G}(t)$ is defined in \eqref{equation_limit_process_G_sample_mean_process} and $\lambda \in \RR^{p}$. The left-hand side of \eqref{eq:fddsamplemaen} can be written as
\begin{equation*}
	\lambda' A^{-1}_{N}(H) S_{\lfloor Nt \rfloor}
=	\sum_{j\in \ZZ}  \sum_{k=1}^{p}
	\sum_{n=1}^{\lfloor Nt \rfloor}  \lambda_{k} 
	a^{-1}_{k}(N) \sum_{l=1}^{p}  \psi_{kl,n-j} \varepsilon_{l,j}
=:	\sum_{l=1}^{p} \sum_{j\in \ZZ} B_{l,Nj}(t) \varepsilon_{l,j},
\end{equation*}
where
\begin{equation*}
a_{k}(N)=
\begin{cases}
N^{\frac{1}{2}+d_{k}},	\hspace{0.2cm}&\text{if }	k \in \{1,\dots,p_{1} \}	,	\\
N^{\frac{1}{2}},	 			\hspace{0.2cm}&\text{if }	k \in \{p_{1}+1,\dots,p \}.
\end{cases}
\end{equation*}
For the convergence of the finite-dimensional distributions, consider
\begin{equation*}
\begin{aligned}
&	
\Big( \sum_{l=1}^{p} \sum_{j \in \ZZ} B_{l,Nj}(t_{1}) \varepsilon_{l,j},\dots,
\sum_{l=1}^{p} \sum_{j \in \ZZ} B_{l,Nj}(t_{z}) \varepsilon_{l,j} \Big)	\\
& =			
\sum_{j \in \ZZ} 
\begin{pmatrix}
B_{1,Nj}(t_{1})	&\dots		&  B_{p,Nj}(t_{1})\\
\vdots			&\ddots	&	\vdots \\
B_{1,Nj}(t_{z})	&\dots		& B_{p,Nj}(t_{z})
\end{pmatrix}						
\begin{pmatrix}
\varepsilon_{1,j}\\
\vdots	\\
\varepsilon_{p,j}
\end{pmatrix}
=:		
\sum_{j \in \ZZ}	B_{Nj} \varepsilon_{j},
\end{aligned}
\end{equation*}
where $t_{1},\dots, t_{z} \in [0,1]$ and $z \in \NN$. This representation allows us to proceed as in the proof of Theorem \ref{Theorem_general_multivariate_linear_process}.
\begin{lemma} \label{le:sequences}
The sequence of matrices $(B_{Nj})_{N \geq 1,j \in \ZZ}$ satisfies the conditions \eqref{conditions_Suquet_Rack}.
\end{lemma}
		
\begin{proof}
The operator norm of the matrix $B_{Nj}$ is given by
\begin{equation*}
\| B_{Nj} \|_{op}= \max_{1 \leq l \leq p} \sum_{i=1}^{z}
| B_{l,Nj}(t_{i}) |.
\end{equation*}
It is enough to prove the statement for $| B_{l,Nj}(t) |$ for all $l \in \{1,\dots,p\}$ and $t \in [0,1]$. By using the inequality \eqref{inequality_omega},
\begin{equation*}
\frac{| \sum_{n=1}^{\lfloor Nt \rfloor} \psi_{kl,n-j} |^2}
{ \omega^{2}_{kl}(\lfloor Nt \rfloor)}
\leq		
\frac{4}{\omega_{kl}(\lfloor Nt \rfloor)}
\left( \frac{\sum_{j \in \ZZ} \psi_{kl,j}^2}{\omega_{kl}(\lfloor Nt \rfloor)}
+ \Big(\sum_{j \in \ZZ} \psi_{kl,j}^2\Big)^{2} \right).
\end{equation*}
Then,
\begin{equation*}
\begin{aligned}
&	
\sup_{j \in \ZZ} |B_{l,Nj}(t)|
=			
\sup_{j \in \ZZ} | \sum_{n=1}^{\lfloor Nt \rfloor}  \sum_{k=1}^{p} \lambda_{k} a^{-1}_{k}(N)	\psi_{kl,n-j} |
\\ & \leq		
4\sum_{k=1}^{p} |\lambda_{k}| |a^{-1}_{k}(N) \omega^{\frac{1}{2}}_{kl}(N) |
\left( \frac{\omega_{kl}(\lfloor Nt \rfloor)}{\omega_{kl}(N)} \right) ^{\frac{1}{2}} 
\left( \frac{\sum_{j \in \ZZ} \psi_{kl,j}^2}{\omega_{kl}(\lfloor Nt \rfloor)}
				+ \Big( \sum_{j \in \ZZ} \psi_{kl,j}^2 \Big)^{\frac{1}{2}} \right) ^{\frac{1}{2}}	
	\to		0,
\end{aligned}
\end{equation*}
since $|a^{-1}_{k}(N) \omega^{\frac{1}{2}}_{kl}(N) | $ converges to zero for each 
$i,k,l \in \{1,\dots,p\}$ by Lemma \ref{Lemma_convergence_to_zero}.
Moreover, for $\omega^{2}_{kl}(N)$ given in \eqref{Definition_small_omega}, 
\begin{equation*}
\frac{\omega^{2}_{kl}(\lfloor Nt \rfloor)}{\omega^{2}_{kl}(N)}
\end{equation*}
is bounded for the long- as well as for the short-range dependent components by Lemma \ref{Lemma_cross_cov_mixture_sample_mean}. The second condition of Lemma \ref{Lemma_Suquet_Rack} follows also by Lemma \ref{Lemma_cross_cov_mixture_sample_mean}, since
\begin{equation*}
\begin{aligned}
			&	\limsup_{N \to \infty} \sum_{j \in \ZZ} |B_{l,Nj}(t)|^2	
	=			\limsup_{N \to \infty} \sum_{j \in \ZZ} | \sum_{n=1}^{\lfloor Nt \rfloor}  
				\sum_{k=1}^{p} \lambda_{k}  a^{-1}_{k}(N) \psi_{kl,n-j} |^2	\\	
 & 	\leq 			\limsup_{N \to \infty} \sum_{j \in \ZZ} \sum_{l=1}^{p}| 
				\sum_{n=1}^{\lfloor Nt \rfloor}  \sum_{k=1}^{p} \lambda_{k} 
				a^{-1}_{k}(N)\psi_{kl,n-j} |^2	
	=			\limsup_{N \to \infty}	\lambda' E| A^{-1}_{N}(H) 
				S_{\lfloor Nt \rfloor} |^2	\lambda
	< \infty.
\end{aligned}
\end{equation*}
\end{proof}
By Lemma \ref{le:sequences} the process $A_{N}^{-1}(H)S_{\lfloor Nt \rfloor}$ can be treated as linear with Gaussian innovations. By \citet[Chapter I, Section 3, Theorem 4]{GikhmanSkorokhod1969} it suffices to establish the componentwise convergence behavior of the cross-covariances as we did in Lemma \ref{Lemma_cross_cov_mixture_sample_mean}. So, the sample mean process converges to the multivariate Gaussian process $\{\mathcal{G}(t)\}_{t \in [0,1]}$ 
defined in \eqref{equation_limit_process_G_sample_mean_process}.

\subsubsection{Tightness} \label{se:SampleMeanTightness}
By \citet[Lemma 1]{BaiTaqqu}, it suffices to prove tightness of each component. Each component is a sum of a sample mean process of univariate linear processes,
\begin{equation*}
S_{\lfloor Nt \rfloor}=
\sum_{l=1}^{p} \sum_{n=1}^{\lfloor Nt \rfloor} \sum_{j \in \ZZ} 
\psi_{kl,j} \varepsilon_{l,n-j}.
\end{equation*}
Since sums of tight processes are tight (see \cite{suquet1999tightness}), what remains is to prove tightness for the sample mean process of a univariate linear process
\begin{equation*}
\sum_{n=1}^{\lfloor Nt \rfloor} \sum_{j \in \ZZ} 
\psi_{kl,j} \varepsilon_{l,n-j}.
\end{equation*}
For the long-range dependent components, this follows by Proposition 4.4.2 in \cite{giraitis} and for the short-range dependent ones by Proposition 4.4.4 in \cite{giraitis} and the assumption $\E\| \varepsilon_{0} \|^{2+\delta}<\infty$ for some $\delta>0$.

\subsection{Proof of Theorem \ref{Theorem_Mixture_SRD_LRD}}
As in the previous section, we first give some auxiliary results regarding the limit processes covariance structure (Section \ref{s:AR_SampleA}). We then investigate the convergence of the finite-dimensional distributions and tightness (Sections \ref{s:fdd_SampleA} and \ref{s:Tight_SampleA}), which establish Theorem \ref{Theorem_Mixture_SRD_LRD}. In Section \ref{s:Properties_SampleA} we investigate the properties of the process $\{Z(t)\}_{t\in \RR}$ defined in \eqref{equality_limit_process_non_Gaussian}.

\subsubsection{Auxiliary results} \label{s:AR_SampleA}
Lemma \ref{Lemma_Gaussian_convergence_covariance} provides the limiting covariance structure for the components which are short-range dependent.
\begin{lemma} \label{Lemma_Gaussian_convergence_covariance}
Let $\widehat{\Gamma}_{N,\ell}^{\operatorname{S}}(t)$ be defined by \eqref{eq:sampleSL}. Then, for $\ell_{1},\ell_{2} \geq 0$ and $t,u \in [0,1]$
\begin{equation*}
\lim_{N \to \infty} N\Cov(\widehat{\Gamma}^{\operatorname{S}}_{N,\ell_{1}}(t),
\widehat{\Gamma}^{\operatorname{S}}_{N,\ell_{2} }(u)) 
=	
\Cov(G^{\operatorname{S}}_{\ell_{1}}(t),G^{\operatorname{S}}_{\ell_{2}}(u))	,
\end{equation*}					
where the right-hand side is given in \eqref{equality_covariance_Gaussian_Ishort}.
\end{lemma}
\begin{proof}
First, note that
\begin{equation} \label{eq:epsiloncomb}
\E(\vecop(\varepsilon_{i_{1}}\varepsilon_{i_{2}}')
(\vecop(\varepsilon_{i_{3}}\varepsilon_{i_{4}}'))')
= 
\begin{cases}
\sigma^{*}	,					
&\hspace{0.2cm} i_{1}=i_{2}=i_{3}=i_{4},	\\
\vecop(I_{p})	(\vecop(I_{p})	)',			
&\hspace{0.2cm} i_{1}=i_{2} \neq i_{3}=i_{4},	\\
I_{p^2},		
&\hspace{0.2cm} i_{1}=i_{3} \neq i_{2}=i_{4},	\\
K_{p},		
&\hspace{0.2cm} i_{1}=i_{4} \neq i_{2}=i_{3},	\\
0,																			
&\hspace{0.2cm} i_{1} \neq i_{2}\neq i_{3}\neq i_{4}
\end{cases}
\end{equation}
with $\sigma^{*}
:=
\E(\vecop(\varepsilon_{i}\varepsilon_{i}')
(\vecop(\varepsilon_{i}\varepsilon_{i}'))')$. We investigate the covariances
\begin{equation*}
\begin{aligned}
&	
\Cov(\widehat{\Gamma}^{\operatorname{S}}_{N,\ell_{1}}(t),
\widehat{\Gamma}^{\operatorname{S}}_{N,\ell_{2}}(u))	\\
&  	= 				
\E\left( \frac{1}{N^2} \sum_{n=1}^{\lfloor Nt \rfloor} 
\sum_{l=1}^{\lfloor Nu \rfloor} 
\vecop_{I_{\operatorname{S}}}(X_{n} X_{n+\ell_{1}}')
(\vecop_{I_{\operatorname{S}}}(X_{l} X_{l+\ell_{2}}'))' \right) -
\frac{\lfloor Nt \rfloor \lfloor Nu \rfloor}{N^2} 
\Gamma^{\operatorname{S}}_{\ell_{1}} (\Gamma^{\operatorname{S}}_{\ell_{2}})'.
\end{aligned}
\end{equation*}
Setting $r=l-n$ and by \eqref{eq:epsiloncomb}, we have
\begin{equation*}
\begin{aligned}
&		
\E(\vecop_{I_{\operatorname{S}}}(X_{n} X_{n+\ell_{1}}')
(\vecop_{I_{\operatorname{S}}}(X_{l} X_{l+\ell_{2}}'))')	\\
& =	
E_{I_{\operatorname{S}}} \sum_{i_{1},i_{2},i_{3},i_{4} \in \ZZ}
(\VPsi_{i_{2}+\ell_{1}} \otimes \VPsi_{i_{1}})
\E(
\vecop(\varepsilon_{n-i_{1}} \varepsilon_{n-i_{2}} ')
(\vecop(\varepsilon_{n-i_{3}} \varepsilon_{n-i_{4}}'))' )
(\VPsi_{i_{4}+r} \otimes \VPsi_{i_{3}+r+\ell_{2}})'
E_{I_{\operatorname{S}}}'	\\
& =			
 \sum_{i \neq j} ( 
\vecop_{I_{\operatorname{S}}}(\VPsi_{i} \VPsi_{i+\ell_{1}}')
(\vecop_{I_{\operatorname{S}}}(\VPsi_{j+r+\ell_{2}} \VPsi_{j+r}'))'+
E_{I_{\operatorname{S}}} (\VPsi_{i+\ell_{1}}\VPsi_{i+r}' \otimes  \VPsi_{j}\VPsi_{j+r+\ell_{2}}') 
E_{I_{\operatorname{S}}}'
\\ & \quad +
E_{I_{\operatorname{S}}} K_{p} (\VPsi_{i}  \VPsi_{i+r}' \otimes 
\VPsi_{j+\ell_{1}}  \VPsi_{j+r+\ell_{2}}') E_{I_{\operatorname{S}}}'
+	
E_{I_{\operatorname{S}}} \sum_{i\in \ZZ} 
(\VPsi_{i+\ell_{1}} \otimes \VPsi_{i}) \sigma^{*} 
(\VPsi_{i+r} \otimes \VPsi_{i+r+\ell_{2}} )' 
E_{I_{\operatorname{S}}}'
\\
& =		
\Gamma^{\operatorname{S}}_{\ell_{1}} (\Gamma^{\operatorname{S}}_{\ell_{2}})'+
E_{I_{\operatorname{S}}}(\Gamma_{r+\ell_{2}} \otimes \Gamma_{r-\ell_{1}})E_{I_{\operatorname{S}}}'+
E_{I_{\operatorname{S}}} K_{p} (\Gamma_{r} \otimes \Gamma_{r+\ell_{2}-\ell_{1}}) 
E_{I_{\operatorname{S}}}' 
\\&\quad	+	
E_{I_{\operatorname{S}}} \sum_{i\in \ZZ} 
(\VPsi_{i+\ell_{1}} \otimes \VPsi_{i}) (\sigma^{*} - \vecop(I_{p})	(\vecop(I_{p})	)' - I_{p^2} - K_{p})
(\VPsi_{i+r} \otimes \VPsi_{i+r+\ell_{2}} )' E_{I_{\operatorname{S}}}'.
\end{aligned}
\end{equation*}
Interchanging the order of summation, assuming $t<u$ and defining $m_{N}=\min(\lfloor Nt \rfloor,\lfloor Nu \rfloor-\lfloor Nt \rfloor)$, we get
\begin{equation} \label{eq:shortCovTr}
\begin{aligned}
\lim_{N \to \infty} 
N\Cov(\widehat{\Gamma}^{\operatorname{S}}_{N,\ell_{1}}(t),
\widehat{\Gamma}^{\operatorname{S}}_{N,\ell_{2}}(u))	
& =	
 \frac{\lfloor Nt \rfloor}{N}\sum_{|r|<\lfloor Nt \rfloor}\left( 1-\frac{|r|}{\lfloor Nt \rfloor}\right) 
T_{r} + \sum_{r=1}^{m_{N}-1}\frac{r}{N}T_{r}
\\ & +
\frac{m_{N}}{N} \sum_{r=m_{N}}^{\lfloor Nu \rfloor -m_{N}}T_{r}
+ \frac{\lfloor Nu \rfloor}{N}\sum_{r=\lfloor Nu \rfloor-m_{N}+1}^{\lfloor Nu \rfloor-1}\left( 1-\frac{r}{\lfloor Nu \rfloor}\right)T_{r},
\end{aligned}
\end{equation}
where $T_{r}=E_{I_{\operatorname{S}}} \widetilde{T}_{r} E_{I_{\operatorname{S}}}'$ with
\begin{equation} \label{eq:Ttilde}
\widetilde{T}_{r}=(\Gamma_{r+\ell_{2}} \otimes \Gamma_{r-\ell_{1}})+K_{p} (\Gamma_{r} \otimes \Gamma_{r+\ell_{2}-\ell_{1}}) + \sum_{i\in \ZZ} 
(\VPsi_{i+\ell_{1}} \otimes \VPsi_{i}) \Sigma
(\VPsi_{i+r} \otimes \VPsi_{i+r+\ell_{2}} )'
\end{equation}
and $\Sigma$ defined in \eqref{equation_big_Sigma}.
In particular, $T_{r}$ is absolutely summable, since
\begin{equation*}
\sum_{r\in \ZZ} \sum_{i\in \ZZ}	
|E_{I_{\operatorname{S}}} (\VPsi_{i+\ell_{1}} \otimes \VPsi_{i}) \Sigma
(\VPsi_{i+r} \otimes \VPsi_{i+r+\ell_{2}} )' E_{I_{\operatorname{S}}}'|<\infty
\end{equation*}
and
\begin{equation*}
\sum_{r\in \ZZ} E_{I_{\operatorname{S}}}(\Gamma_{r} \otimes \Gamma_{r})E_{I_{\operatorname{S}}}'<\infty
\end{equation*}
componentwise. The latter inequalities hold, since $d_{k_{1}}+d_{k_{2}}+d_{k_{3}}+d_{k_{4}}<1$ by construction.
Applying the dominated convergence theorem yields
\begin{equation*}
\begin{aligned}
\lim_{N \to \infty} N\Cov(\widehat{\Gamma}^{\operatorname{S}}_{N,\ell_{1}}(t),
\widehat{\Gamma}^{\operatorname{S}}_{N,\ell_{2}}(u))	
& =
\lim_{N \to \infty} \frac{\lfloor Nt \rfloor}{N}
\sum_{|r|<\lfloor Nt \rfloor}T_{r} \\
& =
t E_{I_{\operatorname{S}}} \Bigg(  \sum_{r \in \ZZ} \Big( 
\Gamma_{r+\ell_{2}} \otimes \Gamma_{r-\ell_{1}}+
K_{p} (	\Gamma_{r} \otimes \Gamma_{r+\ell_{2}-\ell_{1}}) \Big) 
\\ & \quad ~ \phantom{t E_{I_{\operatorname{S}}} \Bigg( )}
+  \sum_{r\in \ZZ} \sum_{i\in \ZZ} 
(\VPsi_{i+\ell_{1}} \otimes \VPsi_{i}) \Sigma
(\VPsi_{i+r} \otimes \VPsi_{i+r+\ell_{2}} )'	 \Bigg) E_{I_{\operatorname{S}}}' .
\end{aligned}
\end{equation*}
\end{proof}

The following lemma deals with the normalization for the long-range dependent components and gives the covariance structure of $Z^{\operatorname{L}}(t)$.
\begin{lemma} \label{Lemma_non_Gaussian_cov}  
Let $\widehat{\Gamma}_{N,\ell}^{\operatorname{L}}(t)$ be defined by \eqref{eq:sampleSL}
and $\Delta_{N} = E_{I_{\operatorname{L}}}( N^{D-\frac{1}{2}I_{p}} \otimes N^{D-\frac{1}{2}I_{p}} )E_{I_{\operatorname{L}}}'$. Then, for $\ell_{1}, \ell_{2} \geq 0$,
\begin{equation*}
\begin{aligned}
&\lim_{N \to \infty} \Cov(\Delta_{N}^{-1} \widehat{\Gamma}^{\operatorname{L}}_{N,\ell_{1}}(t),
\Delta_{N}^{-1} \widehat{\Gamma}^{\operatorname{L}}_{N,\ell_{2}}(u))	
\\ & =
E_{I_{\operatorname{L}}} (t^{\widetilde{D}}   I(R)  t^{\widetilde{D}}
+
u^{\widetilde{D}}   I(R)'  u^{\widetilde{D}}
-|t-u|^{\widetilde{D}}   I(R,t-u)  |t-u|^{\widetilde{D}})E_{I_{\operatorname{L}}}'
\end{aligned}
\end{equation*}	
with $\widetilde{D}=	(D_{p_{1}} \oplus D_{p_{1}})-\frac{1}{2}I_{p_{1}^2}$ 
and $I(R)=\int_{0}^{1} (1-x)x^{\widetilde{D}}(R \otimes R) x^{\widetilde{D}}dx $, where $\oplus$ denotes the Kronecker sum defined as $D_{p_{1}} \oplus D_{p_{1}}=(I_{p_{1}} \otimes D_{p_{1}}) + (D_{p_{1}} \otimes I_{p_{1}})$ and $I(R,t)= I(R) \mathds{1}_{\{t>0\}}+ I(R)' \mathds{1}_{\{t<0\}}$.
\end{lemma}

\begin{proof}
Note that one gets for $\Cov(\Delta_{N}^{-1} \widehat{\Gamma}^{\operatorname{L}}_{N,\ell_{1}}(t),\Delta_{N}^{-1} \widehat{\Gamma}^{\operatorname{L}}_{N,\ell_{2}}(u))$ with $t<u$ the same expression
as in \eqref{eq:shortCovTr} by replacing the subscript ``S'' by ``L'' and adjusting the normalization sequence such that $T_{r}= \frac{1}{N} \Delta_{N}^{-1}E_{I_{\operatorname{L}}} \widetilde{T}_{r} E_{I_{\operatorname{L}}}' \Delta_{N}^{-1}$ with $\widetilde{T}_{r}$ as in \eqref{eq:Ttilde}.
We consider the summands separately. By Proposition \ref{Proposition_mLRD}, the autocovariances of the underlying process $\{X^{\operatorname{L}}_{n}\}_{n \in \ZZ}$ satisfy \eqref{equality_Definition_KP_multivariate_stationary_lrd} with \eqref{equation_Rij} and we get for the first summand of $T_{r}$  		
\begin{align*}
&
\lim_{N \to \infty} \frac{1}{N^2} \sum_{r=0}^{\lfloor Nt \rfloor}\left(\lfloor Nt \rfloor- r \right)
\Delta_{N}^{-1}E_{I_{\operatorname{L}}} (\Gamma_{r+\ell_{2}} \otimes \Gamma_{r-\ell_{1}})E_{I_{\operatorname{L}}}'\Delta_{N}^{-1}	
\\& =
\lim_{N \to \infty} \frac{1}{N^2} \sum_{r=0}^{\lfloor Nt \rfloor}\left(\lfloor Nt \rfloor- r \right)
\Delta_{N}^{-1}
( (r+\ell_{2})^{D_{p_{1}}-\frac{1}{2}I_{p_{1}}} R(r+\ell_{2}) (r+\ell_{2})^{D_{p_{1}}-\frac{1}{2}I_{p_{1}}} \\
& \quad \otimes 
(r-\ell_{1})^{D_{p_{1}}-\frac{1}{2}I_{p_{1}}} R(r-\ell_{1}) (r-\ell_{1})^{D_{p_{1}}-\frac{1}{2}I_{p_{1}}} ) 
\Delta_{N}^{-1}
\\& =
\int_{0}^{t} (t-x) (x^{D_{p_{1}}-\frac{1}{2}I_{p_{1}}} \otimes x^{D_{p_{1}}-\frac{1}{2}I_{p_{1}}})(R \otimes R) (x^{D_{p_{1}}-\frac{1}{2}I_{p_{1}}} \otimes x^{D_{p_{1}}-\frac{1}{2}I_{p_{1}}})dx
\\& =
t^{(D_{p_{1}} \oplus D_{p_{1}})-\frac{1}{2}I_{p_{1}}}
\int_{0}^{1} (1-x)x^{(D_{p_{1}} \oplus D_{p_{1}})-I_{p_{1}}}(R \otimes R) x^{(D_{p_{1}} \oplus D_{p_{1}})-I_{p_{1}}}dx ~
t^{(D_{p_{1}} \oplus D_{p_{1}})-\frac{1}{2}I_{p_{1}}}.
\end{align*}
The second term of $T_{r}$ can be dealt with analogously.
We consider the last summand componentwise for indices taking values in $I_{\operatorname{L}}$. Define
\begin{equation*}
\VPsi^{i}(k;l)
:=		
\VPsi^{i}(k_{1},\dots,k_{4};l_{1},\dots,l_{4})
:=			
\psi_{k_{1}l_{1},i}
\psi_{k_{2}l_{2},i+\ell_{1}} 
\psi_{k_{3}l_{3},i+r+\ell_{2}} 
\psi_{k_{4}l_{4},i+r} 
\end{equation*}
as part of the component, and consider
\begin{equation*}
\sum_{i\in \ZZ}\VPsi^{i}(k;l)
=		
\sum_{i=-\infty}^{-r-1}\VPsi^{i}(k;l) + 
\sum_{i=0}^{\infty}\VPsi^{i}(k;l) + 
\sum_{i=-r}^{0}\VPsi^{i}(k;l).
\end{equation*}	
For example, for the last term, note that
\begin{eqnarray*}
&&	
\sum_{i=-r}^{0} \VPsi^{i}(p;q) 	
\\ && = 		
\sum_{i=-r}^{0}
C_{p_{1}q_{1}}(i)C_{p_{2}q_{2}}(i+k_{1}) 
C_{p_{3}q_{3}}(i+r+k_{2})C_{p_{4}q_{4}}(i+r)	
\\ &&	\phantom{\sum_{i=-r}^{0} C_{ll_{1}}(i)C_{ml_{2}}(i+h)}
\times	
|i|^{d_{p_{1}}-1} |i+k_{1}|^{d_{p_{2}}-1}  
|i+r+k_{2}|^{d_{p_{3}}-1}	|i+r|^{d_{p_{4}}-1}
\\ && 
\sim 	\sum_{i=0}^{r}	
C_{p_{1}q_{1}}(-i)C_{p_{2}q_{2}}(k_{1}-i) 
C_{p_{3}q_{3}}(r+k_{2}-i)C_{p_{4}q_{4}}(r-i)
i^{d_{p_{1}}+d_{p_{2}}-2}  (r-i)^{d_{p_{3}}+d_{p_{4}}-2}	
\\ && =		
r^{d_{p_{1}}+d_{p_{2}}+d_{p_{3}}+d_{p_{4}}-3} 
\sum_{i=0}^{r}	
C_{p_{1}q_{1}}(-i)C_{p_{2}q_{2}}(k_{1}-i) 
C_{p_{3}q_{3}}(r+k_{2}-i)C_{p_{4}q_{4}}(r-i)	
\\ &&	\phantom{r^{1-d_{l}-2d_{m}-d_{\widetilde{l}}} 
\sum_{i=0}^{r}	C_{ll_{1}}(-i)L_{ml_{2}}(h-i)}
\times	
\left( \frac{i}{r}\right) ^{d_{p_{1}}+d_{p_{2}}-2}
\left(\frac{r-i}{r}\right) ^{d_{p_{3}}+d_{p_{4}}-2} \frac{1}{r}	
\\ && \sim	
r^{d_{p_{1}}+d_{p_{2}}+d_{p_{3}}+d_{p_{4}}-3} 
\alpha_{p_{1}q_{1}}^{-}\alpha_{p_{2}q_{2}}^{-}
\alpha_{p_{3}q_{3}}^{+}\alpha_{p_{4}q_{4}}^{+}
\int_{0}^{1}	x^{d_{p_{1}}+d_{p_{2}}-2} (1-x)^{d_{p_{3}}+d_{p_{4}}-2} dx,
\end{eqnarray*}		
as $r \to \infty$. The first and second terms yield similarly
\begin{equation*}
\begin{aligned}
\sum_{i=-\infty}^{-r-1} \VPsi^{i}(k;l)
&	\sim		
r^{d_{k_{1}}+d_{k_{2}}+d_{k_{3}}+d_{k_{4}}-3} 
\alpha_{k_{1}l_{1}}^{-}\alpha_{k_{3}l_{2}}^{-}
\alpha_{k_{3}l_{3}}^{-}\alpha_{k_{4}l_{4}}^{-}
\int_{1}^{\infty}	x^{d_{k_{1}}+d_{k_{2}}-2} (x-1)^{d_{k_{3}}+d_{k_{4}}-2} dx\\
\sum_{i=0}^{\infty}\VPsi^{i}(k;l)
& \sim		
r^{d_{k_{1}}+d_{k_{2}}+d_{k_{3}}+d_{k_{4}}-3} 
\alpha_{k_{1}l_{1}}^{+}\alpha_{k_{3}l_{2}}^{+}
\alpha_{k_{3}l_{3}}^{+}\alpha_{k_{4}l_{4}}^{+}
\int_{0}^{\infty}	x^{d_{k_{1}}+d_{k_{2}}-2} (x+1)^{d_{k_{3}}+d_{k_{4}}-2}  dx.
\end{aligned}
\end{equation*}			
Then,
\begin{equation*}
\sum_{i\in \ZZ}\VPsi^{i}(k;l) 
\sim	
r^{d_{k_{1}}+d_{k_{2}}+d_{k_{3}}+d_{k_{4}}-3}	
C(d_{k_{1}},d_{k_{2}},d_{k_{3}},d_{k_{4}})
\end{equation*}
and
\begin{equation*}
\begin{aligned}
&
\lim_{N \to \infty} \frac{1}{N^2} \sum_{|r|<N}\left(N-|r| \right)
\sum_{i\in \ZZ} N^{2-d_{k_{1}}-d_{k_{2}}-d_{k_{3}}-d_{k_{4}}} \\
&
\hspace{2cm}
\times
\psi_{k_{1}l_{1},i}
\psi_{k_{2}l_{2},i+\ell_{1}} 
\psi_{k_{3}l_{3},i+r+\ell_{2}} 
\psi_{k_{4}l_{4},i+r} 
\Sigma_{k_{1}k_{2}k_{3}k_{4}}
=0,
\end{aligned}
\end{equation*}
where $\Sigma_{k_{1}k_{2}k_{3}k_{4}}$ denotes a component of $\Sigma$ in \eqref{equation_big_Sigma}.	
\end{proof}

The next lemma gives the covariance between the long- and the short-range dependent components of the sample autocovariances.

\begin{lemma} \label{le:mixedcov}
Let $\widehat{\Gamma}_{N,\ell}^{\operatorname{L}}(t)$ and $\widehat{\Gamma}_{N,\ell}^{\operatorname{S}}(t)$ be defined by \eqref{eq:sampleSL} and $\Delta_{N} = E_{I_{\operatorname{L}}}( N^{D-\frac{1}{2}I_{p}} \otimes N^{D-\frac{1}{2}I_{p}} )E_{I_{\operatorname{L}}}'$. Then, for $\ell_{1},\ell_{2} \geq 0$, 
\begin{equation*}
\lim_{N \to \infty} \Cov(\Delta_{N}^{-1} \widehat{\Gamma}^{\operatorname{L}}_{N,\ell_{1}}(t),
N^{\frac{1}{2}} \widehat{\Gamma}^{\operatorname{S}}_{N,\ell_{2}}(u))	=0
\end{equation*}		
\end{lemma}

\begin{proof}
We follow the proof of Lemma \ref{Lemma_Gaussian_convergence_covariance}.  
Note that one gets for $\Cov(\Delta_{N}^{-1} \widehat{\Gamma}^{\operatorname{L}}_{N,\ell_{1}}(t),
N^{\frac{1}{2}} \widehat{\Gamma}^{\operatorname{S}}_{N,\ell_{2}}(u))$ with $t<u$ the same expression
as in \eqref{eq:shortCovTr} by setting 
$T_{r}= N^{-\frac{1}{2}} \Delta_{N}^{-1}E_{I_{\operatorname{L}}} \widetilde{T}_{r} E_{I_{\operatorname{S}}}' $ with $\widetilde{T}_{r}$ as in \eqref{eq:Ttilde}.
Then, since $ \sum_{i=1}^{4} d_{k_{i}}<1$ for $k_{1},k_{2} \in I_{\operatorname{L}}$ and $k_{3},k_{4} \in I_{\operatorname{S}}$,
\begin{equation*}
\sum_{r\in \ZZ} \sum_{i\in \ZZ}	
|E_{I_{\operatorname{L}}} (\VPsi_{i+\ell_{1}} \otimes \VPsi_{i}) \Sigma
(\VPsi_{i+r} \otimes \VPsi_{i+r+\ell_{2}} )' E_{I_{\operatorname{S}}}'|<\infty
\end{equation*}
and
\begin{equation*}
\sum_{r\in \ZZ} E_{I_{\operatorname{L}}}(\Gamma_{r} \otimes \Gamma_{r})E_{I_{\operatorname{S}}}'<\infty
\end{equation*}
componentwise. This implies that $T_{r}$ is absolutely summable over $r$ and
\begin{equation*}
\begin{aligned}
\lim_{N \to \infty} \Cov(\Delta_{N}^{-1} \widehat{\Gamma}^{\operatorname{L}}_{N,\ell_{1}}(t),
N^{\frac{1}{2}} \widehat{\Gamma}^{\operatorname{S}}_{N,\ell_{2}}(u))	
& =	\lim_{N \to \infty} \frac{\Delta_{N}^{-1}}{N^{\frac{1}{2}}} \frac{\lfloor Nt \rfloor}{N} \sum_{|r|<\lfloor Nt \rfloor}\left( 1-\frac{|r|}{\lfloor Nt \rfloor} \right) E_{I_{\operatorname{L}}} \widetilde{T}_{r} E_{I_{\operatorname{S}}}'
= 0.
\end{aligned}
\end{equation*}
\end{proof}

We conclude the section with a comment on the properness of the limiting process
$((Z^{\operatorname{L}}(t))' , (G_{\ell}^{\operatorname{S}}(t))')'$ defined in \eqref{equation_ZL_Z_reduced_to_L} and \eqref{equality_covariance_Gaussian_Ishort}, since it is in parts a consequence of the previous Lemmas \ref{Lemma_Gaussian_convergence_covariance} and \ref{le:mixedcov}.
Since $Z^{\operatorname{L}}(t)$ and $G_{\ell}^{\operatorname{S}}(t)$ are uncorrelated by Lemma \ref{le:mixedcov}, it is enough to prove properness for each of those processes separately.
The matrix $\Lambda:=\sum_{j \in \ZZ} (\psi_{kl,j})_{k=p_{1}+1,\dots,p;l=1,\dots,p}$ is assumed to have full rank by condition (S), so 
$\Lambda \Lambda'$ is positive definite.
The matrix $\Lambda \Lambda'$ is also equal to $\sum_{n \in \ZZ} \Gamma_{p_{2},n}$. For this reason,
$\E ( G^{\operatorname{S}}(t) (G^{\operatorname{S}}(t))' )$ is positive definite, which follows by Lemma \ref{Lemma_Gaussian_convergence_covariance}. 
We refrain from using Lemma \ref{Lemma_non_Gaussian_cov} to infer that $Z^{\operatorname{L}}(t)$ is proper. Instead, we calculate
the covariances of $Z^{\operatorname{L}}(t)$ directly
\begin{equation*}
\begin{aligned}
\E  Z^{\operatorname{L}}(t) (Z^{\operatorname{L}}(t))' 
& =
E_{I_{\operatorname{L}}} \E ( I_{2}(f_{t,D})I'_{2}(f_{t,D}) ) E_{I_{\operatorname{L}}}' 
\\ &= 
E_{I_{\operatorname{L}}} \sum_{s_{1},s_{2} \in \{+,-\}} \sum_{r_{1},r_{2} \in \{+,-\}}
\int_{\RR^2}'  \int_{0}^{t} \int_{0}^{t} ((v_{1}-x_{2})^{D-I_{p}}_{s_{2}} \otimes (v_{1}-x_{1})^{D-I_{p}}_{s_{1}}) 
\\ & ~~~~ \times
\Big((M^{s_{2}}(M^{r_{2}})' \otimes  M^{s_{1}}(M^{r_{1}})')
+
(M^{s_{2}}(M^{r_{1}})' \otimes  M^{s_{1}}(M^{r_{2}})') \Big)
\\ & ~~~~ \times
((v_{2}-x_{2})^{D-I_{p}}_{r_{2}} \otimes (v_{2}-x_{1})^{D-I_{p}}_{r_{1}})
dv_{1}dv_{2} dx_{1} dx_{2} E_{I_{\operatorname{L}}}',
\end{aligned}
\end{equation*}
where we used Theorem 2.2 and 7.9 in \cite{MagnusNeudecker}.
Note that $M^{+}=((A^{+})' ~~ 0_{p \times p_{2}})'$ and $M^{-}=((A^{-})' ~~ 0_{p \times p_{2}})'$. Since the elimination matrix $E_{I_{\operatorname{L}}}$ is applied from both sides, the zero rows and columns are eliminated (see Remark \ref{Remark3}). For this reason, the function is positive definite, since $A^{+}(A^{+})'$ and $A^{-}(A^{-})'$ are positive definite as a consequence of condition (L).

\subsubsection{Convergence of the finite-dimensional distributions} \label{s:fdd_SampleA}
The proof of the convergence of the finite-dimensional distributions is structured as follows. First, we focus on the pure long-range dependence part with $k,l \in I_{\operatorname{L}}$. Note that the sample autocovariances can be separated into the diagonal and off-diagonal parts
	\begin{equation} \label{eq:decomp_off_diagonal}
	\begin{aligned}
			(\widehat{\Gamma}_{N,\ell}-\Gamma_{\ell})(t)
	&=		\frac{1}{N} \sum_{n=1}^{\lfloor Nt \rfloor} \sum_{j+\ell \neq i} 
			\VPsi_{j} \varepsilon_{n-j} \varepsilon_{n+k-i}' \VPsi_{i} '
		+	\frac{1}{N} \sum_{n=1}^{\lfloor Nt \rfloor}\sum_{j\in \ZZ} 
			\VPsi_{j} (\varepsilon_{n-j} \varepsilon_{n-j}'-I) \VPsi_{j+\ell}'		\\
	&=:		O_{N,\ell}(t)+D_{N,\ell}(t).
	\end{aligned}
	\end{equation}
The following Lemma \ref{Lemma_convergence_diagonal} will imply that the sought convergence can be proved only for the off-diagonal part of the long-range dependent components.
We will then provide a convergence result, Lemma \ref{Lemma_off_diagonals}, for multivariate second-order linear forms with respect to the components of a multivariate i.i.d.\ process. Finally, a series of lemmas, Lemma \ref{Lemma_truncated_iid} - \ref{Lemma_non_truncated}, follow the idea of \citet[p.\ 2480]{BaiTaqqu2013} to prove the joint convergence of $Y_{N,\ell}^{\operatorname{L}}(t)$ and $Y_{N,\ell}^{\operatorname{S}}(t)$.

\begin{lemma}	\label{Lemma_convergence_diagonal}
Let $\{X_{n}\}_{n \in \ZZ}$ be as in Theorem \ref{Theorem_Mixture_SRD_LRD}. Then
	\begin{equation*}
	N^{\frac{1}{2}}(\vecop_{I_{\operatorname{L}}}(D_{N,\ell}(t)), \ell=0,\dots,L) 
	\overset{\mathcal{L}}{\longrightarrow} 
	(C_{k}(G(t)), \ell=0,\dots,L), \hspace{0.2cm} t \in [0,1],
	\end{equation*}
in $D[0,1]^{|I_{\operatorname{L}}|}$, where $\{G(t)\}_{t \in [0,1]}$ is an $\RR^{p \times p}$-valued Brownian motion and $C_{\ell}:\RR^{p \times p} \to \RR^{|I_{\operatorname{L}}|}$ with 
$C_{\ell}(X)=\vecop_{I_{\operatorname{L}}}\left( \sum_{j \in \ZZ}  \VPsi_{j} X \VPsi_{j+\ell}' \right) $, 
$X \in \RR^{p \times p}$.
\end{lemma}
\begin{proof}
Consider the linear combination $\sum_{\ell=0}^{L}\mu_{\ell}(\vecop_{I_{\operatorname{L}}}(D_{N,\ell}(t))$
of the diagonal term for $\mu_{\ell} \in \RR, \ell \in \{0,\dots,L \}$. It is short-range dependent in the sense that
\begin{equation*}
\sum_{j \in \ZZ} \| a_{j} \|_{F} < \infty,
\end{equation*}
where $a_{j}:\RR^{p \times p} \to \RR^{p \times p}$ with 
$a_{j}(\cdot)= \sum_{\ell=0}^{L}\mu_{\ell} \vecop_{I_{\operatorname{L}}}(\VPsi_{j} (\cdot) \VPsi_{j+\ell}') $.
Using the Cram\'{e}r-Wold theorem and the results in \citet[Theorem 5]{peligrad}, we get
\begin{equation*}
\sum_{\ell=0}^{L}\mu_{\ell}(\vecop_{I_{\operatorname{L}}}(D_{N,\ell}(t))
\overset{\mathcal{L}}{\longrightarrow} 
AG(t), \hspace{0.2cm} t \in [0,1],
\end{equation*}
for all $\mu_{\ell} \in \RR, \ell \in \{0,\dots,L \}$, where $A=\sum_{j\in \ZZ} a_{j}$.
\end{proof}

To investigate the asymptotic behavior of the off-diagonal terms for the long-range dependent components denoted by $O_{N,\ell}^{\operatorname{L}}(t)=\vecop_{I_{\operatorname{L}}}(O_{N,\ell}(t))$, we prove
\begin{equation} \label{eq: conv}
\Delta_{N}^{-1} O_{N,\ell}^{\operatorname{L}}(t)
\overset{f.d.d.}{\longrightarrow}
Z^{\operatorname{L}}(t).
\end{equation}
The process $Z^{\operatorname{L}}(t)$ is defined in \eqref{equation_ZL_Z_reduced_to_L}.
Rewriting the left-hand side of \eqref{eq: conv} yields
\begin{align*}
\Delta_{N}^{-1} O_{N,\ell}^{\operatorname{L}}(t) 
&= 	
\Delta_{N}^{-1}
\vecop_{I_{\operatorname{L}}}
\Big(\frac{1}{N} \sum_{n=1}^{\lfloor Nt \rfloor} \sum_{j+\ell \neq i} 
\VPsi_{j} \varepsilon_{n-j} \varepsilon_{n+\ell-i}' \VPsi_{i} ' \Big)	\\
&= 	
\Delta_{N}^{-1} E_{I_{\operatorname{L}}}
\vecop
\Big(\frac{1}{N} \sum_{n=1}^{\lfloor Nt \rfloor} \sum_{i_{1} \neq i_{2}} 
\VPsi_{n-i_{1}} \varepsilon_{i_{1}} \varepsilon_{i_{2}}' \VPsi_{n+\ell-i_{2}} ' \Big)	\\
&= 	 
\sum_{i_{1} \neq i_{2}} 
\Delta_{N}^{-1} E_{I_{\operatorname{L}}}
\frac{1}{N} \sum_{n=1}^{\lfloor Nt \rfloor}
(\VPsi_{n+\ell-i_{2}} \otimes \VPsi_{n-i_{1}}) \vecop(\varepsilon_{i_{1}} \varepsilon_{i_{2}}') 	\\
&=		
\sum_{i_{1} \neq i_{2}} C_{N}(i_{1},i_{2}) 
\vecop(\varepsilon_{i_{1}} \varepsilon_{i_{2}}')	,
\end{align*}
where
\begin{equation}	\label{equality_Cn(q,...)}
C_{N}(i_{1},i_{2})
=		
\Delta_{N}^{-1} E_{I_{\operatorname{L}}}
\frac{1}{N} \sum_{n=1}^{\lfloor Nt \rfloor}
(\VPsi_{n+\ell-i_{2}} \otimes \VPsi_{n-i_{1}}).
\end{equation}
The following lemma provides a generalization of Proposition 14.3.2 in \cite{giraitis}.
It uses the space of simple functions $S_{M}(\RR^2, \RR^{p^2 \times p^2})$ defined as follows. Partition the space $\RR^2$ into cubes of size $\frac{1}{M}$ with $M\in \NN$.
Let $(\Delta):=\Delta_{1} \times \Delta_{2} \subset \RR^2$ be such that 
$\Delta_{1}, \Delta_{2} \in I_{M}$, where 
$I_{M}:= \{ (\frac{j}{M}, \frac{j+1}{M} ], j \in \ZZ \}, M\in \NN$. We write 
$(\Delta) \in \{\Delta_{M} \}$ and 
$(\Delta) \in \{\Delta_{M}^{\operatorname{diag}} \}$, if $\Delta_{1}=\Delta_{2}$. 
This space $S_{M}(\RR^2, \RR^{p^2 \times p^2})$ then consists of $\RR^{p^2 \times p^2}$-valued functions $f = f(x_{1},x_{2})$ on $\RR^2$ satisfying 
\begin{equation*}
f(x_{1},x_{2})=
\begin{cases}
f^{\Delta_{1} \Delta_{2}},			
\hspace{0.2cm}&	x \in (\Delta), (\Delta) \in \{\Delta_{M} \}, \\
0,											
\hspace{0.2cm}&	x \in (\Delta), (\Delta) \in \{\Delta_{M}^{\operatorname{diag}} \} ,
\end{cases}
\end{equation*}
where $f^{\Delta_{1} \Delta_{2}} \in \RR^{p \times p}$. Set $\Vert f \Vert^2= \int_{\RR^{2}} \| f(x_{1},x_{2}) \|^2_{F} dx_{1}dx_{2}$ for $f:\RR^2 \to \RR^{p^2 \times p^2}$.
\begin{lemma} \label{Lemma_off_diagonals}
Consider the off-diagonal tuple
\begin{equation} \label{eq:Q2}
Q_{2}(C_{N})=\sum_{i_{1}\neq i_{2}} 
C_{N}(i_{1},i_{2}) \vecop(\varepsilon_{i_{1}} \varepsilon_{i_{2}}') .
\end{equation}
Assume that the weights $C_{N}$ are such that the functions
\begin{equation*}
\widetilde{C}_{N}(x_{1},x_{2})= N C_{N}([x_{1}N],[x_{2}N]), 
\hspace{0.2cm}x_{1}, x_{2} \in \RR
\end{equation*}
satisfy
\begin{equation*}
||\widetilde{C}_{N}-f|| \to 0
\end{equation*}
for a function
$f \in L^{2}(\RR^{2}, \RR^{p^2 \times p^2})$. Then $Q_{2}(C_{N}) \overset{f.d.d.}{\longrightarrow} I_{2}(f)$.
\end{lemma}

\begin{proof}
Let $f_{\varepsilon}$ be in $S_{M}(\RR^2, \RR^{p \times p})$ and define
\begin{equation*}
C_{N,\varepsilon}(i_{1},i_{2}) 
:=		
N^{-1} f_{\varepsilon}\left( \frac{i_{1}}{N},\frac{i_{2}}{N} \right) , 
\hspace*{0.2cm}
i_{1},i_{2} \in \ZZ.
\end{equation*}
It is enough to prove that for all $\varepsilon>0$, there exists $f_{\varepsilon} \in S_{M}(\RR^{2}, \RR^{p^2 \times p^2})$, $M \geq 1$, 
such that
\begin{align}
\Var \|Q_{2}(C_{N})-Q_{2}(C_{N,\varepsilon}) \|_{F} \leq \varepsilon,	\label{Con1}		\\
\Var \| I_{2}(f_{\varepsilon})-I_{2}(f) \|_{F} \leq \varepsilon, 	\label{Con2}	\\	
Q_{2}(C_{N,\varepsilon}) \overset{f.d.d.}{\longrightarrow} I_{2}(f_{\varepsilon}) ,	\label{Con3}
\end{align}
as $N \to \infty$. Note that
\begin{equation*}
\E(\vecop(\varepsilon_{i_{1}} \varepsilon_{i_{2}}')
(\vecop(\varepsilon_{j_{1}} \varepsilon_{j_{2}}'))')=
\begin{cases}
I_{p^2},		&\text{if } i_{1}=j_{1}, i_{2}=j_{2}, \\
K_{p},		&\text{if } i_{1}=j_{2}, i_{2}=j_{1}, \\
0,			&\text{otherwise. }
\end{cases}
\end{equation*}
Then, for \eqref{Con1}, 
\begin{equation*}
\begin{aligned}
\Var \|Q_{2}(C_{N})\|_{F}
&=			
\E \| \sum_{i_{1} \neq i_{2}} 
C_{N}(i_{1},i_{2}) 
\vecop(\varepsilon_{i_{1}} \varepsilon_{i_{2}}')	\|_{F}^2
\leq			
2 \sum_{i_{1} \neq i_{2}} \| C_{N}(i_{1},i_{2}) \|_{F}^2	\\
&=			
2 \int_{\RR^2}'  N^2  
\| C_{N}([x_{1}N],[x_{2}N]) \|_{F}^2 dx_{1} dx_{2}
=		
2 \| \widetilde{C}_{N} \|^2 .
\end{aligned}
\end{equation*}
This implies
\begin{equation*}
\Var \|Q_{2}(C_{N})-Q_{2}(C_{N,\varepsilon})\|_{F}
\leq			
2 \| \widetilde{C}_{N} - \widetilde{C}_{N,\varepsilon}\|^2.
\end{equation*}
The latter bound could be approximated by finding simple functions $f_{\varepsilon}$ such that
\begin{equation*}
\| \widetilde{C}_{N} - \widetilde{C}_{N,\varepsilon}\|^2 < \varepsilon
\end{equation*}
as $N \to \infty$. By assumption, there is $N_{0} \geq 1$ such that
\begin{equation*}
\| \widetilde{C}_{N} -\widetilde{C}_{N_{0}}\|^2	
\leq			
2\| \widetilde{C}_{N} -f \|^2+
2\| f-\widetilde{C}_{N_{0}}\|^2
\leq		
\frac{\varepsilon}{6}
\end{equation*}
for all $N \geq N_{0}$.
Given $N_{0}\geq 1$ and $\varepsilon>0$, there exist simple functions 
$f_{\varepsilon}$ such that
\begin{equation*}
\| \widetilde{C}_{N_{0}} -f_{\varepsilon} \|^2 \leq \varepsilon/6.
\end{equation*}
The function $\widetilde{C}_{N,\varepsilon}$ derived from $C_{N,\varepsilon}$ satisfies
\begin{equation*}
\|f_{\varepsilon}-\widetilde{C}_{N,\varepsilon}\|^2
=		
\int_{\RR^2} \Big\| f_{\varepsilon}(x_{1},x_{2})-
f_{\varepsilon}\Big(\frac{\lfloor x_{1}N \rfloor }{N} , 
\frac{\lfloor x_{2}N \rfloor }{N}\Big)\Big\|_{F}^2 dx_{1}dx_{2}
\to  0
\end{equation*}
as $N \to \infty$.
Hence, there is $\widetilde{N}_{0} \geq 1$ such that
\begin{equation*}
\| \widetilde{C}_{N} -\widetilde{C}_{N,\varepsilon} \|^2
\leq	
3\| \widetilde{C}_{N} -\widetilde{C}_{N_{0},\varepsilon}\|^2+
3\| \widetilde{C}_{N_{0}} -\widetilde{C}_{\varepsilon}\|^2+
3\| \widetilde{C}_{\varepsilon} -\widetilde{C}_{N,\varepsilon}\|^2
\leq		
\varepsilon /2.
\end{equation*}
This proves \eqref{Con2} since
\begin{equation*}
\Var \| I(f_{\varepsilon})-I(f) \|	
\leq 	
2 \|f_{\varepsilon}-f \|^2	
\leq	
2 (2\|f_{\varepsilon}-\widetilde{C}_{N_{0}}\|^2 + 2\|\widetilde{C}_{N_{0}}-f \|^2).
\end{equation*}
Finally, for \eqref{Con3}, note that
\begin{equation*}
\begin{aligned}
Q_{2}(C_{\varepsilon,N})
&=				
\sum_{i_{1} \neq  i_{2}} 
C_{N,\varepsilon}(i_{1},i_{2}) 
\vecop(\varepsilon_{i_{1}} \varepsilon_{i_{2}}')
=				
\sum_{i_{1} \neq i_{2}} N^{-1} 
f_{\varepsilon} \Big(\frac{i_{1}}{N},\frac{i_{2}}{N} \Big) 
\vecop(\varepsilon_{i_{1}} \varepsilon_{i_{2}}')		\\
&=				
\sum_{(\Delta) \in \{\Delta_{M}\}} 
f_{\varepsilon}^{\Delta_{1} \Delta_{2}} N^{-1} 
\sum_{i_{1} \neq i_{2}}
\vecop(\varepsilon_{i_{1}} \varepsilon_{i_{2}}')
\mathds{1} _{ \{ \frac{i_{1}}{N} \in \Delta_{1}, 
\frac{i_{2}}{N} \in \Delta_{2} \}} \\
&=				
\sum_{(\Delta) \in \{\Delta_{M}\}} 
f_{\varepsilon}^{\Delta_{1} \Delta_{2}} 
\vecop(W_{N} (\Delta_{1}) W'_{N} (\Delta_{2})),
\end{aligned}
\end{equation*}		
where
\begin{equation} \label{equality_partial_sum}
W_{N}(\Delta_{i})
=	
N^{-\frac{1}{2}} \sum_{j : \frac{j}{N}\in \Delta_{i}} \varepsilon_{j}.
\end{equation}
Now, define the vector 
$W(\Delta_{i}):=(W_{1,N}(\Delta_{i}),\dots,W_{p,N}(\Delta_{i}))$.
Since the intervals $\Delta_{i}$ are disjoint, $(W(\Delta_{i}))_{i \in \ZZ}$ are 
independent random vectors. Since $\{\varepsilon_{j}\}_{ j \in \ZZ}$ are i.i.d.\,
the central limit theorem applies and hence
\begin{equation*}
(W_{N}(\Delta_{-J}),\dots,W_{N}(\Delta_{J}))
\overset{f.d.d.}{\longrightarrow}
(W(\Delta_{-J}),\dots,W(\Delta_{J})).
\end{equation*}			 		
Using the  continuous mapping theorem yields
\begin{equation*}
Q_{2}(C_{N,\varepsilon}) 
\overset{f.d.d.}{\longrightarrow} 
\sum_{(\Delta) \in \{\Delta_{M}\}} 
f_{\varepsilon}^{\Delta_{1} \Delta_{2}} 
\vecop(W (\Delta_{1}) W' (\Delta_{2}))
= I_{2}(f_{\varepsilon}).
\end{equation*}
\end{proof}

As noted above, we prove Theorem \ref{Theorem_Mixture_SRD_LRD} through a number of lemmas following the idea of \citet[p.\ 2480]{BaiTaqqu2013}.
We introduce the $\kappa$-truncated versions of the quantities of interest
\begin{equation*}
X_{n}^{(\kappa)} = 
\sum_{j=-\kappa}^{\kappa} \VPsi_{j}\varepsilon_{n-j}, 
\hspace*{2cm} 	
\widehat{\Gamma}_{N,\ell}^{\operatorname{S},(\kappa)}
=	
\vecop_{I_{\operatorname{S}}}\Big(\frac{1}{N} \sum_{n=1}^{N} X_{n}^{(\kappa)} X_{n+\ell}^{(\kappa)'}\Big)
\end{equation*}
and
\begin{equation*}
\Gamma_{\ell}^{\operatorname{S},(\kappa)}
=		
\vecop_{I_{\operatorname{S}}}(\E(X_{0}^{(\kappa)}X_{\ell}^{(\kappa)'}))
=		
\vecop_{I_{\operatorname{S}}}\Big(\sum_{j=-\kappa-\ell}^{\kappa-\ell} \VPsi_{j} \VPsi_{j+\ell}'\Big).
\end{equation*}
The $\kappa$-truncated version of 
$Y_{N,\ell}(t)$ is written as
\begin{equation*}
Y_{N,\ell}^{(\kappa)}(t)
=	
\frac{1}{N}\sum_{n=1}^{\lfloor Nt \rfloor} 
\left( X_{n}^{(\kappa)} X_{n+\ell}^{(\kappa)'}-\E(X_{0}^{(\kappa)} X_{\ell}^{(\kappa)'}) \right).
\end{equation*}
The truncated version of the limit process $\{G^{\operatorname{S}}_{\ell}(t)\}_{t \in [0,1]}$ is defined by its covariance structure
\begin{equation} 	\label{equality_covariance_truncated}
\begin{aligned}
\Cov(G^{\operatorname{S},(\kappa)}_{\ell_{1}}(t),G^{\operatorname{S},(\kappa)}_{\ell_{2}}(u))
& =		
\min(t,u)E_{I_{\operatorname{S}}} \Bigg( 
\sum_{r \in \ZZ} \Big( \Gamma_{r+\ell_{2}}^{(\kappa)} \otimes \Gamma_{r-\ell_{1}}^{(\kappa)}+
K_{p} (\Gamma_{r}^{(\kappa)} \otimes \Gamma_{r+\ell_{2}-\ell_{1}}^{(\kappa)}	) 
\Big) 	
\\ & \hspace{1cm}+
\sum_{r \in \ZZ} \sum_{i=-\kappa}^{\kappa} 
(\VPsi_{i+\ell_{1}} \otimes \VPsi_{i}) \Sigma
(\VPsi_{i+r} \otimes \VPsi_{i+r+\ell_{2}} )'	 \Bigg) E_{I_{\operatorname{S}}}'.	
\end{aligned} 
\end{equation}

\begin{lemma} \label{Lemma_truncated_iid}
Suppose the assumptions of Theorem \ref{Theorem_Mixture_SRD_LRD} and set
\begin{equation*}
Y_{N,\ell}^{\operatorname{S},(\kappa)}(t)=\vecop_{I_{\operatorname{S}}}(Y_{N,\ell}^{(\kappa)}(t))
\hspace*{0.2cm} \text{ and } \hspace{0.2cm}
W_{N}(t)=N^{-\frac{1}{2}} \sum_{n=1}^{\lfloor Nt \rfloor} 
\varepsilon_{n}
\end{equation*}
with $\varepsilon_{0} \in \RR^{p}$. Then,
\begin{equation*}
\vecop(N^{\frac{1}{2}}Y_{N,\ell}^{\operatorname{S},(\kappa)}(t), W_{N}(t))
\overset{f.d.d.}{\longrightarrow}
\vecop(G^{\operatorname{S},(\kappa)}_{\ell}(t),W(t)),
\end{equation*}
where $G^{\operatorname{S},(\kappa)}_{\ell}(t)$ is defined by \eqref{equality_covariance_truncated} and $W(t)$ is a standard Brownian motion.
\end{lemma}
	
\begin{proof}
By the Cram\'{e}r-Wold theorem, we can prove that
\begin{equation}	\label{equality_truncated_cramer}
\mu' N^{\frac{1}{2}} Y_{N,\ell}^{\operatorname{S},(\kappa)}(t)+
\nu' W_{N}(t)
\overset{f.d.d.}{\longrightarrow}
\mu' G^{\operatorname{S},(\kappa)}_{\ell}(t)+\nu'W(t)
\end{equation}
with $\mu \in \RR^{|I_{\operatorname{S}}|}$ and $\nu \in \RR^{p}$.
The left-hand side of \eqref{equality_truncated_cramer} can be written as
\begin{equation*}
\mu' N^{\frac{1}{2}} Y^{\operatorname{S},(\kappa)}_{N,\ell}(t)+
\nu' W_{N}(t)
=
N^{-\frac{1}{2}} \sum_{n=1}^{\lfloor Nt \rfloor} Q_{n,\ell}^{(\kappa)}
\end{equation*}
with
\begin{equation*}
Q_{n,\ell}^{(\kappa)}
=	
\mu' \vecop_{I_{\operatorname{S}}}( X_{n}^{(\kappa)} X_{n+\ell}^{(\kappa)'}-\E(X_{0}^{(\kappa)} X_{\ell}^{(\kappa)'}) )
+	
\nu' \sum_{i=(2\kappa+\ell-1)n+1}^{(2\kappa+\ell)n} \varepsilon_{i} .
\end{equation*}
Following \cite{HorvathKokoszka} and \cite{BrockwellDavis}, we shall use the notion of 
$m$-dependence. Recall that a stationary sequence $\{Y_{j}\}_{j \in \ZZ}$ is $m$-dependent, where $m$ is a non-negative integer if the random sequences $\{Y_{j}\}_{j \leq 0}$ and $\{Y_{j}\}_{j \geq m+1}$ are independent.
Define the sequence $\{Y_{n}\}_{n \in \ZZ}$ of $\RR^{p^2(\ell+1)}$-valued random variables by	
\begin{equation*}
Y_{n}=	(Z_{n},Z_{n+1},\dots ,Z_{n+\ell}),
\end{equation*}
where $Z_{n+\ell}:=\vecop_{I_{\operatorname{S}}}(X_{n}^{(\kappa)} X_{n+\ell}^{(\kappa)'})$.
Since the process $\{Y_{n}\}_{n \in \ZZ}$ is $(2\kappa+\ell)$-dependent, so is $\{Q_{n,\ell}^{(\kappa)}\}_{n \in \ZZ}$.
For any $\lambda \in \RR^{\ell+1}$, the sequence $(\lambda'Q_{n,\ell}^{(\kappa)})$ is $(2\kappa+\ell)$-dependent as well. Then, the convergence \eqref{equality_truncated_cramer} follows by the functional central limit theorem 
for $m$-dependent processes in \cite{Billingsley_Inv}. Since the  joint asymptotic normality is proven, it is left to verify that the asymptotic covariance structure of $\vecop(N^{\frac{1}{2}} Y_{N,\ell}^{S,(\kappa)}(t), W_{N}(t))$ coincides with that of $\vecop(G^{\operatorname{S},(\kappa)}_{\ell}(t),W(t))$. 
For $\Cov(N^{\frac{1}{2}}Y_{N,\ell}^{\operatorname{S},(\kappa)}(t),N^{\frac{1}{2}}Y_{N,\ell}^{\operatorname{S},(\kappa)}(u))$, the relation follows by Lemma \ref{Lemma_Gaussian_convergence_covariance} and \eqref{equality_covariance_truncated}. By similar arguments as in Lemma 
\ref{Lemma_Gaussian_convergence_covariance} for $t<u$,
\begin{equation} \label{eq:correlated}
\begin{aligned}
&
\lim_{N \to \infty} \E(N^{\frac{1}{2}} Y_{N,\ell}^{\operatorname{S},(\kappa)}(t)W_{N}'(u)) \\
&=	
\lim_{N \to \infty} \frac{\lfloor Nt \rfloor}{N}
\sum_{|r|<\lfloor Nt \rfloor} \Big(1-\frac{|r|}{\lfloor Nt \rfloor} \Big)
\sum_{i=-\kappa}^{\kappa} E_{I_{\operatorname{S}}} (\VPsi_{r+\ell-i} \otimes \VPsi_{r-i})\widetilde{\Sigma}	\\
&=	
t \sum_{ r \in \ZZ} \sum_{i=-\kappa}^{\kappa} E_{I_{\operatorname{S}}} 
(\VPsi_{r+\ell-i} \otimes \VPsi_{r-i})\widetilde{\Sigma},
\end{aligned}
\end{equation}
where $\widetilde{\Sigma}=\E(\vecop(\varepsilon_{0}\varepsilon_{0} ')\varepsilon_{0}' )$.
\end{proof}
	
Since $N^{\frac{1}{2}} Y_{N,\ell}^{\operatorname{S},(\kappa)}(t)$ and $W_{N}'(t)$ are not asymptotically uncorrelated as shown in \eqref{eq:correlated}, one can infer that the resulting limits $Z^{\operatorname{L}}(t)$ of the long-range dependent components and $G_{\ell}^{\operatorname{S}}(t)$ of the short-range dependent components in Theorem \ref{Theorem_Mixture_SRD_LRD} are not independent. However, Lemma \ref{le:mixedcov} gives uncorrelatedness between $Z^{\operatorname{L}}(t)$ and $G_{\ell}^{\operatorname{S}}(t)$.
\par	
Using the same notation as in the proof of Lemma \ref{Lemma_off_diagonals}, the joint convergence in the previous lemma still holds by replacing $W_{N}(t)$ by 
$(W_{N}(\Delta_{-J}),\dots,W_{N}(\Delta_{J}))$ with 
$W_{N}(\Delta_{i}):=(W_{1,N}(\Delta_{i}),\dots,W_{p,N}(\Delta_{i}))$ and 
$W_{N}(\Delta_{i})$ as in \eqref{equality_partial_sum},
since the intervals $\Delta_{i}$ are disjoint.

\begin{lemma}	\label{Lemma_truncated_linearform}
Replace $(N^{\frac{1}{2}} Y_{N,\ell}^{\operatorname{S},(\kappa)}(t), W_{N}(t))$ by 
$(N^{\frac{1}{2}} Y_{N,\ell}^{\operatorname{S},(\kappa)}(t), \Delta_{N}^{-1} O^{\operatorname{L}}_{N,\ell}(t))$
in Lemma \ref{Lemma_truncated_iid}.
Assume that the weights $C_{N}$ defined in \eqref{equality_Cn(q,...)} are such that
for a function $f \in L^{2}(\RR^{2}, \RR^{p^2 \times p^2})$ the functions
\begin{equation*}
\widetilde{C}_{N}(x_{1},x_{2})= N C_{N}([x_{1}N],[x_{2}N]), 
\hspace{0.2cm}x_{1}, x_{2} \in \RR
\end{equation*}
satisfy
\begin{equation*}
||\widetilde{C}_{N}-f|| \to 0.
\end{equation*}
Then,
\begin{equation*}
\vecop(N^{\frac{1}{2}}Y_{N,\ell}^{\operatorname{S},(\kappa)}(t), \Delta_{N}^{-1} O^{\operatorname{L}}_{N,\ell}(t))
\overset{f.d.d.}{\longrightarrow}
\vecop(G^{\operatorname{S},(\kappa)}_{\ell}(t),Z^{\operatorname{L}}(t)).
\end{equation*}
\end{lemma}
		
\begin{proof}
We prove the lemma by combining the previous Lemmas \ref{Lemma_truncated_iid} and
\ref{Lemma_off_diagonals}. As in \eqref{equality_Cn(q,...)}, the sum $O^{\operatorname{L}}_{N,\ell}(t)$ can be represented as $Q_{2}(C_{N})$ with $Q_{2}$ defined in \eqref{eq:Q2}.
By Lemma \ref{Lemma_off_diagonals}, for all $\varepsilon>0$, 
there exists $f_{\varepsilon} \in S_{M}(\RR^{2},\RR^{p^2 \times p^2})$, $M \geq 1$, such that \eqref{Con1}, \eqref{Con2} and \eqref{Con3} are satisfied. 
As in the proof of Lemma \ref{Lemma_off_diagonals} by applying the continuous mapping theorem to the result in Lemma \ref{Lemma_truncated_iid}, we get
\begin{equation*}
\vecop(N^{\frac{1}{2}} Y^{\operatorname{S},(\kappa)}_{N,\ell}(t),
Q_{2}(C_{N,\varepsilon}))
\overset{f.d.d.}{\longrightarrow}
\vecop(G^{\operatorname{S},(\kappa)}_{\ell}(t),
I_{2}(f_{\varepsilon})).
\end{equation*}
Now, define
\begin{equation*}
\begin{aligned}
R^{\varepsilon,(\kappa)}_{N,\ell}(t)
&=
N^{\frac{1}{2}} Y_{N,\ell}^{\operatorname{S},(\kappa)}(t)+
Q_{2}(C_{N,\varepsilon}),	\\
R^{(\kappa)}_{N,\ell}(t)
&=
N^{\frac{1}{2}} Y_{N,\ell}^{\operatorname{S},(\kappa)}(t)+
Q_{2}(C_{N}),
\\
R^{\varepsilon,(\kappa)}_{\ell}(t)
&=
G_{\ell}^{\operatorname{S},(\kappa)}(t)+I_{2}(f_{\varepsilon}),	\\
R_{\ell}^{(\kappa)}(t)
&=
G_{\ell}^{\operatorname{S},(\kappa)}(t)+I_{2}(f).
\end{aligned}
\end{equation*}
Then, by \eqref{Con1}, \eqref{Con2} and \eqref{Con3},
\begin{equation*}
\begin{aligned}
&	
R^{\varepsilon,(\kappa)}_{N,\ell}(t)	
\overset{f.d.d.}{\longrightarrow}	
R^{\varepsilon,(\kappa)}_{\ell}(t),	
\hspace{0.2cm} \text{ as } \hspace{0.2cm} N \to \infty, 
\\ &	
R^{\varepsilon,(\kappa)}_{\ell}(t)		
\overset{f.d.d.}{\longrightarrow}	
R_{\ell}^{(\kappa)}(t),	
\hspace{0.2cm} \text{ as } \hspace{0.2cm} \varepsilon \to 0,
\\ &	
\lim_{\varepsilon\to 0} \limsup_{N \to \infty} 
\Var \|R^{\varepsilon,(\kappa)}_{N,\ell}(t)-R^{(\kappa)}_{N,\ell}(t)\|_{F}=0
\hspace{0.2cm} \text{ for all } t \in [0,1],
\end{aligned}
\end{equation*}
which finally implies
\begin{equation*}
R^{(\kappa)}_{N,\ell}(t)\overset{f.d.d.}{\longrightarrow}R^{(\kappa)}_{\ell}(t)
\end{equation*}
by \citet[Lemma 4.2.1]{giraitis}.
\end{proof}

In the following lemma the truncated qunatities get replaced by their non-truncated originals.
\begin{lemma} \label{Lemma_non_truncated}
Replace $(N^{\frac{1}{2}} Y_{N,\ell}^{\operatorname{S},(\kappa)}(t), \Delta_{N}^{-1} O^{\operatorname{L}}_{N,\ell}(t))$ by $(N^{\frac{1}{2}}Y_{N,\ell}^{\operatorname{S}}(t), \Delta_{N}^{-1} O^{L}_{N,\ell}(t))$ in Lemma \ref{Lemma_truncated_linearform}. Then,
\begin{equation*}
\vecop(N^{\frac{1}{2}} Y_{N,\ell}^{\operatorname{S}}(t), \Delta_{N}^{-1} O^{\operatorname{L}}_{N,\ell}(t))
\overset{f.d.d.}{\longrightarrow}
\vecop(G^{\operatorname{S}}_{\ell}(t),Z^{\operatorname{L}}(t)).
\end{equation*}
\end{lemma}
	
\begin{proof}
Define
\begin{equation*}
\begin{aligned}
R_{N,\ell}(t)	
&=	
\mu' N^{\frac{1}{2}} Y_{N,\ell}^{\operatorname{S}}(t)+\lambda' Q_{2}(C_{ N}),	
\\
R_{\ell}(t)		
&=	
\mu' G_{\ell}^{\operatorname{S}}(t)+\lambda' I_{2}(f)	
\end{aligned}
\end{equation*}
for $\mu \in \RR^{|I_{\operatorname{S}}|}$, $\lambda \in \RR^{|I_{\operatorname{L}}|}$. We prove that
\begin{align}
&R_{N,\ell}^{(\kappa)}(t) 	\overset{f.d.d.}{\longrightarrow} 	R_{\ell}^{(\kappa)}(t),	
\hspace{0.2cm} \text{ as } N \to \infty, 	\label{Con1_lemma_non_trunc}\\
&R_{\ell}^{(\kappa)}(t)		\overset{f.d.d.}{\longrightarrow}	R_{\ell}(t),	
\hspace{0.2cm} \text{ as } l \to \infty	\label{Con2_lemma_non_trunc}	,\\
& \lim_{\kappa \to \infty} \limsup_{N \to \infty}  \Var(R_{N,\ell}^{(\kappa)}(t) -R_{N,\ell}(t))=0
\hspace{0.2cm} \text{ for all } t \in [0,1].
\label{Con3_lemma_non_trunc}
\end{align}	
The convergence \eqref{Con1_lemma_non_trunc} follows by Lemma 
\ref{Lemma_truncated_linearform}, \eqref{Con2_lemma_non_trunc} is a consequence of
\begin{equation*}
\lim_{\kappa \to \infty} \Cov(G_{\ell_{1}}^{\operatorname{S},(\kappa)}(t),G^{\operatorname{S},(\kappa)}_{\ell_{2}}(u))
=
\Cov(G_{\ell_{1}}^{\operatorname{S}}(t),G^{\operatorname{S}}_{\ell_{2}}(u)),
\end{equation*}
while \eqref{Con3_lemma_non_trunc} follows from
\begin{equation*}
\begin{aligned}
\lim_{\kappa \to \infty} \limsup_{N \to \infty}
\E(\mu' N^{\frac{1}{2}} Y_{N,\ell}^{\operatorname{S},(\kappa)}(t))^2
&=	
\E(\mu'G_{\ell}^{\operatorname{S}}(t))^2,	\\
\limsup_{N \to \infty}
\E(\mu' N^{\frac{1}{2}} Y_{N,\ell}^{\operatorname{S}}(t))^2
&=	
\E(\mu' G_{\ell}^{\operatorname{S}}(t))^2,	\\
\lim_{\kappa \to \infty} \limsup_{N \to \infty}
\E(\mu' N^{\frac{1}{2}} Y_{N,\ell}^{\operatorname{S},(\kappa)}(t)
\mu' N^{\frac{1}{2}} Y_{N,\ell}^{\operatorname{S}}(t))
&=	
\E(\mu'G_{\ell}^{\operatorname{S}}(t))^2	.
\end{aligned}
\end{equation*}
\end{proof}

To conclude the proof of Theorem \ref{Theorem_Mixture_SRD_LRD}, it remains to verify that $C_{N}$ defined in \eqref{equality_Cn(q,...)} satisfies the assumptions of Lemma \ref{Lemma_off_diagonals}. 
Write
\begin{equation*}
\begin{aligned}
NC_{N}([x_{1}N],[x_{2}N])
&=		
\sum_{n=1}^{\lfloor Nt \rfloor} \Delta_{N}^{-1} E_{I_{\operatorname{L}}}
(\VPsi_{n+\ell-[x_{2}N]} \otimes \VPsi_{n-[x_{1}N]})	\\
&=		
N \int_{0}^{t} \Delta_{N}^{-1} E_{I_{\operatorname{L}}}
(\VPsi_{[vN]+\ell-[x_{2}N]} \otimes \VPsi_{[vN]-[x_{1}N]}) dv.
\end{aligned}
\end{equation*}
Then, considering the expression componentwise
\begin{equation*}
\begin{aligned}
&	
N^{2-d_{l_{1}}-d_{l_{2}}}
\psi_{l_{1}q_{1}, [vN]-[x_{1}N]} \psi_{l_{2}q_{2}, [vN]+\ell-[x_{2}N]}	
\\ &  = 	
N^{2-d_{l_{1}}-d_{l_{2}}} C_{l_{1}q_{1}}([vN]-[x_{1}N]) |[vN]-[x_{1}N]|^{d_{l_{1}}-1}  
\\ &	\phantom{  = l	N^{2-d_{l_{1}}-d_{l_{2}}}}
\times	
C_{l_{2}q_{2}}([vN]+\ell-[x_{2}N]) |[vN]+\ell-[x_{2}N]|^{d_{l_{2}}-1} 
\\ & =	
p_{N}^{(l_{1},q_{1})}(v,x_{1})p_{N}^{(l_{2},q_{2})}(v,x_{2})
\Big[ 
\nu_{l_{1}l_{2}q_{1}q_{2}}^{(+, +)}(v,x_{1},x_{2})+
\nu_{l_{1}l_{2}q_{1}q_{2}}^{(-, -)}(v,x_{1},x_{2})\\
&\phantom{ =  l p_{n}^{(l_{1},q_{1})}(v,x_{1})p_{n}^{(l_{2},q_{2})}(v,x_{2}) \Big[  }
+	
\nu_{l_{1}l_{2}q_{1}q_{2}}^{(+, -)}(v,x_{1},x_{2})+
\nu_{l_{1}l_{2}q_{1}q_{2}}^{(-, +)}(v,x_{1},x_{2})
\Big] ,
\end{aligned}
\end{equation*}		
where
\begin{equation*}
\nu_{l_{1}l_{2}q_{1}q_{2}}^{(s_{1},s_{2})}(v,x_{1},x_{2})
:=		
(v-x_{1})_{s_{1}}^{d_{l_{1}}-1}  
(v-x_{2})_{s_{2}}^{d_{l_{2}}-1}
\alpha_{l_{1}q_{1}}^{s_{1}}\alpha_{l_{2}q_{2}}^{s_{2}}, 
\text{ for } s_{1},s_{2} \in \{+,-\}
\end{equation*}
and
\begin{equation*}
p_{N}^{(l_{i},q_{i})}(v,x_{i})	
:=
\frac{N^{1-d_{l_{i}}}
C_{l_{i}q_{i}}([vN]-[x_{i}N]) |[vN]-[x_{i}N]|^{d_{l_{i}}-1}}{
(v-x_{i})_{+}^{d_{l_{i}}-1}\alpha_{l_{i}q_{i}}^{+}+
(v-x_{i})_{-}^{d_{l_{i}}-1}\alpha_{l_{i}q_{i}}^{-}  }\to 1.
\end{equation*}
Furthermore, there are constants $C_{1}, C_{2}$, such that
\begin{equation*}
\sup_{N \geq 1} \sup_{v,x_{i}} p_{N}^{(l_{i},q_{i})}(v,x_{1})	\leq C_{i}, 
\hspace{0.2cm}
i=1,2,
\end{equation*}
which implies
\begin{equation*}
\begin{aligned}
&\sup_{x_{1},x_{2}} \Big| \int_{0}^{t}
N^{2-d_{l_{1}}-d_{l_{2}}}
\psi_{l_{1}q_{1}, [vN]-[x_{1}N]} \psi_{l_{2}q_{2},[vN]+\ell-[x_{2}N]} dv  \Big| 
\\&\leq		
C\sup_{x_{1},x_{2}}
|\sum_{s_{1},s_{2} \in \{+,-\}}\int_{0}^{t} \nu_{l_{1}l_{2}q_{1}q_{2}}^{(s_{1},s_{2} )}(v,x_{1},x_{2}) dv|
\end{aligned}
\end{equation*}
and by the dominated convergence theorem
\begin{equation*}
\int_{0}^{t}
N^{2-d_{l_{1}}-d_{l_{2}}}
\psi_{l_{1}q_{1},[vN]-[x_{1}N]} \psi_{l_{2}q_{2},[vN]+\ell-[x_{2}N]} dv
\to			
\sum_{s_{1},s_{2} \in \{+,-\}} \int_{0}^{t} \nu_{l_{1}l_{2}q_{1}q_{2}}^{(s_{1},s_{2} )}(v,x_{1},x_{2})dv.
\end{equation*}
Then, applying again the dominated convergence theorem leads to
\begin{equation*}
\begin{aligned}
&	
\| \widetilde{C}_{N} -f_{H,t} \|^2	
\\ &=	
\int_{\RR^2} \Big\| \int_{0}^{t} N \Delta_{N}^{-1} E_{I_{\operatorname{L}}}
(\VPsi_{[vN]+\ell-[x_{2}N]} \otimes \VPsi_{[vN]-[x_{1}N]}) dv 
-	
E_{I_{\operatorname{L}}} f_{H,t}(x_{1},x_{2}) \Big\|^2_{F}	dx_{1}dx_{2}
\to 0.
\end{aligned}
\end{equation*}
This shows that the conditions in Lemma \ref{Lemma_off_diagonals} are satisfied.

\subsubsection{Tightness} \label{s:Tight_SampleA}
\begin{lemma} 
Under the assumptions in Theorem \ref{Theorem_Mixture_SRD_LRD} the sample autocovariance process is tight in $D[0,1]^{p^2}$.
\end{lemma}

\begin{proof}
By Lemma 1 in \cite{BaiTaqqu}, it is enough to prove tightness in each component
\begin{equation*}
\begin{aligned}
&
a^{-1}_{kl}(N)Y_{kl,N,\ell}(t) 
\\ & =
a^{-1}_{kl}(N) \sum_{n=1}^{ \lfloor Nt \rfloor} (X_{k,n}X_{l,n+\ell} - \E(X_{k,0}X_{l,\ell}))
\\ & =
\sum_{r,s=1}^{p} a^{-1}_{kl}(N) \sum_{n=1}^{ \lfloor Nt \rfloor} \Big(
\sum_{j+\ell \neq i }
\psi_{kr,i} \psi_{ls,j} \varepsilon_{r,n-j}\varepsilon_{s,n+\ell-i}
+
\sum_{j \in \ZZ}
\psi_{kr,j} \psi_{ls,j+\ell} (\varepsilon_{r,n-j}\varepsilon_{s,n-j}-1) \Big)
\end{aligned}
\end{equation*}
where
\begin{equation*}
a_{kl}(N)=
\begin{cases}
N^{d_{k}+d_{l}-1}, 	&\text{if } \hspace{0.2cm} k,l \in I_{\operatorname{L}},	\\
N^{\frac{1}{2}},		&\text{if } \hspace{0.2cm}	k,l \in I_{\operatorname{S}}.
\end{cases}
\end{equation*}	 
By \cite{suquet1999tightness} it is enough to prove tightness of one summand
\begin{equation*}
\begin{aligned}
a^{-1}_{kl}(N) \sum_{n=1}^{ \lfloor Nt \rfloor} 
\sum_{j+\ell \neq i }
\psi_{kr,i} \psi_{ls,j} \varepsilon_{r,n-j}\varepsilon_{s,n+\ell-i}
+
a^{-1}_{kl}(N) \sum_{n=1}^{ \lfloor Nt \rfloor} 
\sum_{j \in \ZZ}
\psi_{kr,j} \psi_{ls,j+\ell} (\varepsilon_{r,n-j}\varepsilon_{s,n-j}-1).
\end{aligned}
\end{equation*}	
Note that for fixed $r,s \in \{1,\dots,p\}$ the first summand is the sample mean process of a univariate bilinear polynomial-form process. The second summand is the sample mean process of a univariate linear process generated by an i.i.d.\ sequence $\{\varepsilon_{r,j}\varepsilon_{s,j}\}_{j \in \ZZ}$. 
\par
In the case $k,l \in I_{\operatorname{S}}$ and under the assumption $\E\| \varepsilon_{0} \|^5<\infty$, the first summand is tight by Theorem 3.8 (2.d.) in \cite{BaiTaqqu2013}, 
The second summand is tight by Proposition 4.4.4 in \cite{giraitis}. 
For $k,l \in I_{\operatorname{L}}$, the first summand is tight by Theorem 4.8.2 in \cite{giraitis} and the second by Proposition 4.4.4 in \cite{giraitis}.
\end{proof}

\subsubsection{Properties of the limit process} \label{s:Properties_SampleA}
The next lemma provides some properties of the process $\{Z(t)\}_{t \in \RR}$ defined in \eqref{equality_limit_process_non_Gaussian}.

\begin{lemma} \label{Lemma_properties_Rosen}
The process $\{Z(t)\}_{t \in \RR}$ is operator self-similar with scaling family $\{\Delta_{c}:\RR^{p \times p} \\ \to \RR^{p \times p} ~|~ c>0\}$, where
$\Delta_{c} = c^{H} \otimes c^{H}$, and has stationary increments.
\end{lemma}
\begin{proof}
We get
\begin{align*}
I_{2}(f_{H,ct})
& 
\overset{\phantom{f.d.d.}}{=} 	
(c^{H-I_{p}} \otimes c^{H-I_{p}} ) \sum_{s_{1},s_{2} \in \{+,-\}} \int_{\RR^2}' \int_{0}^{ct}
\Big(\Big(\frac{v-x_{2}}{c} \Big)_{s_{1}}^{H-I_{p}} M^{s_{1}} \otimes
\Big(\frac{v-x_{1}}{c}\Big)_{s_{2}}^{H-I_{p}} M^{s_{2}}\Big)  \\
&	
\phantom{(c^{H-I_{p}} \otimes c^{H-I_{p}} ) \sum_{s_{1},s_{2} \in \{+,-\}} \int_{\RR^2}'}
\times	dv \vecop \Big(W(dx_{1})W'(dx_{2}) \Big)	\\
& \overset{\phantom{f.d.d.}}{=} 	
(c^{H-\frac{1}{2}I_{p}} \otimes c^{H-\frac{1}{2}I_{p}}) \int_{\RR^2} '
f_{H,t}(x_{1},x_{2}) 
\vecop \Big(W(dcx_{1})W'(dcx_{2}) \Big)	  \\
& \overset{f.d.d.}{=} 
(c^{H} \otimes c^{H}) I_{2}(f_{H,t}) ,
\end{align*}
since	 $W(d(cx))\overset{f.d.d.}{=} c^{\frac{1}{2}I_{p}}W(dx)$ and by Theorem 2.2 in \cite{MagnusNeudecker}. Thus, $Z(ct)\overset{f.d.d.}{=} 	(c^{H} \otimes c^{H}) Z(t) $.	\\
Similarly, we can prove that the process has stationary increments, since for any $\ell \in \RR$
\begin{equation*}
\begin{aligned}
I_{2}(f_{H,t+\ell})
& \overset{\phantom{f.d.d.}}{=}	
\sum_{s_{1},s_{2} \in \{+,-\}} \int_{\RR^2}' \int_{-\ell}^{t}
((v-(x_{2}-\ell))_{s_{1}}^{H-I_{p}}M^{s_{1}} \otimes
(v-(x_{1}-\ell))_{s_{2}}^{H-I_{p}}M^{s_{2}}) \\
& \phantom{\sum_{s_{1},s_{2} \in \{+,-\}} \int_{\RR^2}'}
\times
dv \vecop \Big(W(dx_{1})W'(dx_{2}) \Big) \\
& \overset{\phantom{f.d.d.}}{=}	 
\int_{\RR^2}' (f_{H,t}(x_{1},x_{2})-f_{H,-\ell}(x_{1},x_{2}))
\vecop \Big(W(d(x_{1}+\ell))W'(d(x_{2}+\ell))\Big)	\\
&\overset{f.d.d.}{=} 
I_{2}(f_{H,t})-I_{2}(f_{H,-\ell}),		
\end{aligned}
\end{equation*}
since $W(d(x+\ell))\overset{f.d.d.}{=} W(dx)$. Thus 
$I_{2}(f_{H,t+\ell})-I_{2}(f_{H,\ell}) \overset{f.d.d.}{=} I_{2}(f_{H,t})-I_{2}(f_{H,0})$
and $Z(t+\ell)-Z(\ell)\overset{f.d.d.}{=} Z(t)-Z(0)$.
\end{proof}

\bigskip
\noindent {\bf Acknowledgements:} The author would like to thank the three anonymous referees and the two editors for their comments and their advice that led to a substantial revision and improvement of the original version of this paper.
Parts of this work were finalized during a stay in the Department of Statistics and Operation Research at the University of North Carolina, Chapel Hill. 
The author thanks the department for its hospitality and, in particular, Vladas Pipiras for his support.
The author would also like to thank the Research Training Group 2131 - \textit{High-dimensional Phenomena in Probability - Fluctuations and Discontinuity} for financial support. 

\bigskip

\bibliographystyle{plainnat}
\bibliography{biblinearpr}

\end{document}